\documentclass[onefignum,onetabnum]{siamonline250211} 

\usepackage{xcolor} 

\definecolor{darkgreen}{RGB}{0,128,64}
\definecolor{deepbrown}{RGB}{139,69,19}

\usepackage{amsmath,amssymb,amsfonts}
\DeclareMathOperator*{\argmin}{arg\,min}

\newcommand{\R}{\mathbb{R}}

\newcommand{\colw}{0.135\textwidth}

\newcommand{\setfigparams}[3]{%
  \renewcommand{\colw}{#1\textwidth}%
  \setlength{\tabcolsep}{#3}%
}

\newcommand{\presetThree}{\setfigparams{0.20}{0.9}{0.4pt}}

\newcommand{\presetFive}{\setfigparams{0.19}{0.99}{0.4pt}}
\newcommand{\presetSix}{\setfigparams{0.165}{0.99}{0.4pt}}
\newcommand{\presetSeven}{\setfigparams{0.12}{0.99}{0.4pt}}

\newcommand{\localset}[3]{%
  \renewcommand{\colw}{#1\textwidth}%
  \setlength{\tabcolsep}{#3}%
}

\newcommand{\cellimg}[2]{
  \begin{tikzpicture}[baseline=(current bounding box.center)]
    \node[inner sep=0.0] (img) {\includegraphics[width=\linewidth]{#1}};
    \if\relax\detokenize{#2}\relax\else
      \node[anchor=south west, xshift=4pt, yshift=4pt, inner sep=1pt,
            fill=white, font=\scriptsize] at (img.south west) {#2};
    \fi
  \end{tikzpicture}%
}

\newcommand{\cellimgTR}[2]{%
  \begin{tikzpicture}[baseline=(current bounding box.center)]
    \node[inner sep=0.0] (img) {\includegraphics[width=\linewidth]{#1}};
    \if\relax\detokenize{#2}\relax\else
      \node[
        anchor=north west,
        xshift=2pt, yshift=-2pt,     
        inner sep=0.8pt, outer sep=0pt,
        fill=white,
        font=\fontsize{6.6}{7.4}\selectfont  
      ] at (img.north west) {#2};
    \fi
  \end{tikzpicture}%
}

\usepackage{array}
\usepackage{multirow}
\usepackage{float}
\usepackage{placeins}
\usepackage[export]{adjustbox}
\usepackage{cancel}

\usepackage{tikz}
\usetikzlibrary{positioning,shapes,arrows.meta,calc,fit,decorations.pathreplacing}
\usetikzlibrary{shapes.geometric,shadows}
\definecolor{emerald}{RGB}{58,191,153}

\usepackage{lipsum}
\usepackage{amsfonts}
\usepackage{graphicx}
\usepackage{epstopdf}

\usepackage{algorithmic}
\ifpdf
  \DeclareGraphicsExtensions{.eps,.pdf,.png,.jpg}
\else
  \DeclareGraphicsExtensions{.eps}
\fi

\usepackage{enumitem}
\setlist[enumerate]{leftmargin=.5in}
\setlist[itemize]{leftmargin=.5in}

\newsiamremark{remark}{Remark}
\newsiamremark{hypothesis}{Hypothesis}
\crefname{hypothesis}{Hypothesis}{Hypotheses}
\newsiamthm{claim}{Claim}
\newsiamremark{fact}{Fact}
\crefname{fact}{Fact}{Facts}

\headers{Guided Variational Network for Image Decomposition}{A. Lanza, S. Morigi, Y. Wen, and L. Yang}

\title{Guided Variational Network for Image Decomposition
\thanks{\funding{The work of S. Morigi was funded by PNRR CN-HPC, under the NextGenerationEU program, CUP J33C22001170001. The work of A. Lanza and S. Morigi was supported in part by the INdAM--GNCS 2026 projects and by the PRIN2022 project ``Inverse Problems in the Imaging Sciences (IPIS)'' (2022 ANC8HL, CUP J53D23003670006) and the PRIN2022\_PNRR project CUP J53D23014080001. The work of Y. Wen was supported in part by the NNSF of China grant 12361089. The work of L. Yang was supported in part by the Postgraduate Scientific Research Innovation Project of Hunan Province.}}
}

\author{
Alessandro Lanza\thanks{Department of Mathematics, University of Bologna, 40127 Bologna, Italy
(\email{alessandro.lanza2@unibo.it}, \email{serena.morigi@unibo.it}).}
\and
Serena Morigi\footnotemark[2]
\and
Youwei Wen\thanks{School of Mathematics and Statistics, Hunan Normal University, Changsha, Hunan, China
(\email{wenyouwei@gmail.com}).}
\and
Li Yang\thanks{School of Mathematics and Statistics, Hunan Normal University, Changsha, Hunan, China; and Department of Mathematics, University of Bologna, 40127 Bologna, Italy (\email{liyang161029@gmail.com}). Corresponding author.}
}

\usepackage{amsopn}
\providecommand{\diag}{\operatorname{diag}}

\ifpdf
\hypersetup{
  pdftitle={Guided Variational Network for Image Decomposition},
  pdfauthor={A. Lanza, S. Morigi, Y. Wen, and L. Yang}
}
\fi

\begin{document}

\maketitle

\begin{abstract}
Cartoon-texture image decomposition is a critical preprocessing problem bottlenecked by the numerical intractability of classical variational or optimization models and the tedious manual tuning of global regularization parameters.
We propose a Guided Variational Decomposition (GVD) model 
which introduces spatially adaptive quadratic norms whose pixel-wise weights are 
learned either through local probabilistic statistics or via a lightweight neural network 
within a bilevel framework.
This leads to a unified, interpretable, and 
computationally efficient model that bridges classical variational ideas with 
modern adaptive and data-driven methodologies.
Numerical experiments on this framework, which inherently includes automatic parameter selection, delivers GVD as a robust, self-tuning, and superior solution for reliable image decomposition.
\end{abstract}



\section{Introduction}
Image decomposition, in particular the separation of an image into cartoon and texture components, has long been a fundamental problem in image processing and computer vision. The cartoon part, characterized by piecewise smooth structures, provides the geometric backbone of the image, while the texture part captures oscillatory details and fine-scale patterns. A reliable cartoon--texture decomposition not only enhances visual understanding but also serves as a crucial preprocessing step in tasks such as image denoising \cite{LeVese}, compression \cite{Wakin2002DCC}, recognition \cite{YinGoldfarbOsher2005TVL1}, and medical imaging \cite{KohanBehnam2011MedDenoise}.

The {\em cartoon $+$ texture decomposition problem} considered here for two dimensional images aims to split a $h \times w$ vectorized image $f \in \R^{n}$ - with $n = h \times w$ - into two components:
\begin{equation*}
	f \,\;{=}\;\, c + t \,,
\end{equation*}
where $c$ represents the cartoon component containing homogeneous or smoothly varying regions, and $t$ captures texture-like oscillatory structures.
Given the desired properties of $c$ and $t$, a variational decomposition model for a given image $f$ can be formulated as:
\begin{equation}
\{\widehat{c},\widehat{t}\} \in \arg \min_{c,t \in \R^n}
\left\{
 \| c \|_{\star} + \lambda \| t \|_{\small{\square}}
\right\}
\quad \text{subject to} \quad c+t=f,
\label{eq:mod_constrained}
\end{equation}
where $\lambda > 0$ is a regularization parameter, and $\|\cdot\|_{\star}$ and $\|\cdot\|_{\square}$ denote suitable norms (or semi-norms) that encode the structural priors of the cartoon and texture components, respectively. 
Naturally, the hard constraint in \eqref{eq:mod_constrained} is replaced by a quadratic penalty, leading to an unconstrained formulation
\begin{equation}
\left\{\widehat{c},\widehat{t}\,\right\} \in \arg \min_{c,t \in \R^n}
\left\{
\frac{1}{2}\| f-(c+t) \|_2^2
+ \lambda_1 \| c \|_{\star}
+ \lambda_2 \| t \|_{\small{\square}}
\right\},
\label{eq:mod_unconstrained}
\end{equation}
which we adopt in this paper. Here $\lambda_1, \lambda_2>0$ and $\|\cdot\|_2$ denotes the Euclidean norm.

Classical variational models in the form \eqref{eq:mod_constrained} or \eqref{eq:mod_unconstrained}, such as Rudin--Osher--Fatemi (ROF)-type approaches and their extensions, have established the theoretical and computational foundation of cartoon--texture decomposition. The Total Variation (TV) is a natural regularizer for modeling cartoon images \cite{RUDIN1992259}. For zero-mean oscillatory components, Meyer \cite{Meyer} introduced the $\mathrm{G}$-space which is more suitable than the $\mathrm{L}_2$-norm for modeling textures \cite{AABC}. Others proposed negative Sobolev norms as numerically treatable approximations of the $\mathrm{G}$-norm \cite{LeVese,LeVese1,SIAMHKLM}.
A widely used instantiation of \eqref{eq:mod_constrained} is obtained by selecting $\|\cdot\|_{\star}=\|\cdot\|_{\mathrm{TV}}$ for the cartoon component and $\|\cdot\|_{\square}=\|\cdot\|_{\mathrm{G}}$ for the texture component. Concretely, the TV semi-norm is defined as
\begin{equation}
\| c\|_{\mathrm{TV}} = \sum_{i=1}^n \left\| \left(\nabla c\right)_i\right\|_2,
\label{eq:TVnorm}
\end{equation}
where $\nabla$ denotes the discrete gradient operator with $(\nabla c)_i \in \R^2$ the discrete gradient at the pixel $i$, and where the Meyer's $\mathrm{G}$-norm is defined by
\begin{equation}
\| t \|_{\mathrm{G}} = \min_{\xi \in \R^{2n}}
\left\{ \,
\max_{1\leq i \leq n} \left\| \xi_i \right\|_2 : \;\,
t = \mathrm{div}(\xi) = \nabla_x^* \xi_1 + \nabla_y^* \xi_2 \, 
\right\} \, ,
\label{eq:Gnorm}
\end{equation}
where $\mathrm{div} = - \nabla^*$ denotes the discrete divergence operator, defined as the negative adjoint of the gradient operator.

Intrinsic difficulties with those variational models come from the numerical intractability of the considered norms, the tedious and time consuming parameter tuning process, and computational challenges in minimization with non-convex regularization terms. 
In particular, parameter tuning, $\lambda$ in \eqref{eq:mod_constrained} and $\lambda_1$, $\lambda_2$ in \eqref{eq:mod_unconstrained}, highly influences the quality of the obtained decomposition. Most of 
the existing strategies to select model parameters are based on trial-and-error  approaches. Bilevel 
frameworks to automatically select the free model parameters are proposed in \cite{LMS21_IPIP,COAP_2024}, exploiting the noise whiteness property.

In this work, we propose a novel framework, termed \emph{Guided Variational Decomposition}, which introduces weighted energy norms into a simple quadratic variational model, under an automatic parameter selection strategy. The key idea is to preserve the computational efficiency of quadratic formulations while enriching their expressive power through weighted energy norms with spatially-adaptive, pixel-wise weights.
This allows the model to accommodate highly diverse and spatially varying structures in natural images, where smooth background regions and fine oscillatory textures demand different regularization strengths that a global weight cannot capture effectively.

The weight matrices $W$ defining the spatially adaptive energy norms are computed in two different ways: (i) a purely model-based probabilistic method, and
(ii) a data-driven approach based on a convolutional neural network.
These weight matrices are progressively refined across iterations, using feedback from the most recent estimates of the cartoon and texture components. This iterative guidance provides evolving structural cues and enabling flexibly adaptation to heterogeneous image regions while retaining the efficiency of quadratic inner solvers.

The main contributions of this paper are:
\begin{itemize}
    \item We introduce a \emph{Guided Variational Decomposition} (GVD) model: a quadratic variational model for cartoon--texture separation that employs spatially adaptive, pixel-wise weights to reconcile the efficiency of quadratic formulations with the expressive power required for heterogeneous natural images.
    \item We propose two instantiations of the spatially adaptive weight maps: a data-driven supervised variant which couples a scalar multilayer perceptron (MLP)  - for global regularization parameters - and a lightweight U-Net - for pixel-wise weights;
    and a model-based probabilistic estimator that derives weights from local neighborhood statistics and requires no training data.
    \item 
    We develop an end-to-end trainable variational network framework - \emph{Neural Guided Variational Decomposition} (NGVD) - which implements a bilevel learning optimization scheme, and a probabilistic variational framework - \emph{Probabilistic Guided Variational Decomposition} (PGVD) - that, iteratively, alternates between constructing spatially adaptive weight maps and solving the resulting fixed-weight quadratic subproblem. This design preserves numerical stability while enabling progressively refined structural guidance.
    \item We provide a theoretical analysis that places our iteratively guided scheme in a fixed-point framework. Concretely, we prove (i) uniqueness and conditioning of each fixed-weight inner solve, (ii) existence of outer fixed points, (iii) a sufficient, verifiable contractivity condition with explicit constants that ensures linear convergence to a unique fixed point, and (iv) Lipschitz stability bounds with respect to measurement perturbations. 
    \item We perform extensive numerical experiments on synthetic and real images, including ablation studies and comparisons with classical and recent state-of-the-art methods, demonstrating that the proposed framework yields improved decomposition quality, better edge preservation, and practical robustness.
\end{itemize}


The remainder of the paper is organized as follows. Section~\ref{sec:related} reviews some related works. In Section~\ref{modelsect} we introduce and analyze the proposed quadratic variational decomposition model with spatially adaptive weights and discuss its numerical solutions. 
In Section \ref{sec:PGVD} we ground the model via probabilistic arguments and then introduce the Probabilistic Guided Variational Decomposition (PGVD) for parameter selection.
In Section~\ref{sec:NGVD} we provide details on the bilevel optimization approach developed by a Neural Guided Variational Decomposition (NGVD) framework. 
Section~\ref{sec:theory} states the main theoretical results described above; complete proofs and auxiliary lemmas are collected in Appendix~\ref{app:proofs}. Section~\ref{sec:experiments} presents our experimental evaluation, including details of implementation, ablation studies, and comparisons on synthetic and real datasets. Finally, Section~\ref{sec:con} draws conclusions and discusses limitations and directions for future work.

\section{Related Work}\label{sec:related}

Early cartoon--texture decomposition models relied on global regularization parameters within variational formulations. The seminal ROF model \cite{RUDIN1992259} introduced total variation (TV) as an effective prior for cartoon-like structures, while Meyer’s $G$-space \cite{Meyer} provided a dedicated functional setting for oscillatory textures. Subsequent works proposed practical and numerically tractable approximations of the $G$-norm using negative Sobolev metrics \cite{LeVese,LeVese1}, enabling texture extraction through convex or quasi-convex optimization frameworks. In addition, efficient solvers for the original Meyer model have also been studied, e.g., via primal--dual schemes \cite{wen2019primal}. Although these classical approaches form the basis of modern decomposition models, the use of global regularization parameters limits their ability to accommodate the spatial heterogeneity of natural images.

To overcome the shortcomings of global weighting, a wide range of locally adaptive regularizers have been proposed. Spatially varying TV formulations \cite{Chambolle2004,Chen2010TVadaptive} adapt the amount of smoothing according to local geometry or contrast. Patch-based and nonlocal techniques \cite{Buades2005,Gu2014LR} leverage patch recurrence, self-similarity, or low-rank statistics to better separate textures from piecewise-smooth structures. Weighted least-squares approaches \cite{Farbman2008,Wu2010} further incorporate edge-aware metrics to enhance locality and spatial adaptivity. Recently, a method for image decomposition combining a weighted least-squares data term with low-rank regularization was studied in \cite{li2025cartoon}. While these methods greatly enhance flexibility, they often rely on handcrafted descriptors and do not provide pixel-wise regularization weights that can be learned and updated within an automatic pipeline.

Automatic parameter selection has been investigated through bilevel optimization, which provides a rigorous framework for learning optimal regularization parameters from data. Foundational works \cite{KunischPockBilevel} established differentiation through variational models, while their applications to imaging tasks demonstrated the feasibility of learning global regularization strengths \cite{Calatroni2019}. For cartoon--texture decomposition, an adaptive parameter rule exploiting noise whiteness was proposed in \cite{girometti2023ternary}. Nevertheless, most existing bilevel strategies focus on learning a small set of global parameters, and thus remain limited in their ability to capture strong local variability between edges and textures.

More recently, data-driven approaches have introduced implicit forms of spatial adaptivity. Plug-and-play priors \cite{Venkatakrishnan2013,Chan2017Plug} embed CNN-based denoisers within iterative schemes and have been applied to structure--texture modeling \cite{doi:10.1137/24M1677770}, while deep-unfolding architectures such as the Low Patch Rank decomposition network (LPR-Net) \cite{10.1007/978-3-031-92366-1_30} learn local structures by unrolling classical optimization steps. Although powerful, these approaches often do not yield a simple explicit energy with directly interpretable pixel-wise regularization weights, which makes it less straightforward to control or analyze the spatial regularization mechanism.

Existing unrolling and plug-and-play methods typically approximate or replace discrete algorithmic steps, such as proximal operators or denoising sub-problems, with neural networks. 
While computationally powerful, these approaches somewhat sever the explicit connection to an underlying energy functional, rendering rigorous convergence analysis difficult to establish. The proposed NGVD embeds data-driven components (specifically, spatially adaptive weights) directly into the regularization term. This design ensures that the inner solver remains exact (linear system solver). By preserving a well-defined variational structure at every iteration, NGVD facilitates the formal fixed-point convergence framework detailed in Section 5 and verified numerically in Figure 7.

The proposed GVD model bridges classical variational principles and data-driven methods by introducing spatially adaptive quadratic norms. Using a bilevel optimization framework, pixel-wise weights are learned via local statistics or a lightweight CNN.
This automatic, parameter-free approach balances structural guidance with numerical reconstruction, ensuring stability and mutual improvement between weights and results. By maintaining an exact quadratic solver (unlike the approximations used in unrolling or PnP methods), the model remains interpretable, efficient, and mathematically rigorous regarding convergence.

\section{Spatially-adaptive quadratic GVD model}
\label{modelsect}
In this section, first we define and motivate the proposed spatially-adaptive variational model for decomposition as a fully quadratic variant of the unconstrained Meyer's model, then we analyze the model in terms of existence and uniqueness of solution and, finally, we introduce the numerical algorithm adopted \mbox{for its solution.}

The unconstrained Meyer’s model for cartoon/texture decomposition is defined by

\begin{equation}
\{\widehat{c},\widehat{t}\} \in \arg \min_{c,t \in \R^n}
\left\{ \,
\frac{1}{2}\| f-(c+t) \|_2^2
+ \lambda_1 \| c \|_{\mathrm{TV}}
+ \lambda_2 \| t \|_G
\right\},
\label{eq:Meyer_model}
\end{equation}
with the TV semi-norm $\| c \|_{\mathrm{TV}}$ and the G-norm $\| t \|_G$ defined in \eqref{eq:TVnorm} and \eqref{eq:Gnorm}, respectively.

We propose a variant to model \eqref{eq:Meyer_model} which configures as a fully quadratic variational framework. Specifically, we substitute the non-smooth TV semi-norm and $G$-norm with a squared weighted $H^1$ semi-norm and a squared weighted $H^{-1}$ norm, respectively, thereby transforming the original problem into a computationally efficient, strongly convex quadratic task. The proposed model reads
\begin{equation}
\{\widehat{c},\widehat{t}\} \in \arg \min_{c,t \in \R^n}
\left\{ \,
\frac{1}{2}\| f-(c+t) \|_2^2
+ \frac{\lambda_1}{2} \| c \|^2_{H^1(W_1)}
+ \frac{\lambda_2}{2} \| t \|^2_{H^{-1}(W_2)}
\right\},
\label{eq:our0}
\end{equation}
with the squared weighted semi-norm for $c$ and squared weighted norm for $t$ defined by
\begin{equation}
\| c \|^2_{{H^1}(W_1)} 
\,\;{=}\;\, 
\| \nabla c \|_{W_1}^2 \, ,
\label{eq:ourRc}
\end{equation}
\vspace{-0.4cm}
\begin{equation}
\| t \|^2_{H^{-1}(W_2)} 
\,\;{=}\;\,  
\min_{\xi \in \R^{2n}} \left\{ \, 
\| \xi \|^2_{W_2}  : \;\,
t = \mathrm{div}(\xi) = - \left(\nabla_x^* \xi_x + \nabla_y^* \xi_y \right) \, 
\right\} \, ,
\label{eq:ourRt}
\end{equation}
respectively, and where we assume that the two weight matrices $W_1, W_2 \in \mathbb{R}^{2n \times 2n}$ are diagonal and positive definite, in particular
\begin{equation}
W_i \,\;{=}\:\! 
\left[\!\!
\begin{array}{cc}	
W_{i,x} & \!\!\! 0 \\
0       & \!\!\! W_{i,y}
\end{array}
\!\!\right] \! ,
\;\;\:\text{with}\;\;
\left\{ \!\!
\begin{array}{lcl}
W_{i,x} &\!\!\!{=}\!\!\!& \mathrm{diag}\left(\omega_{i,x}\right) 
\\
W_{i,y} &\!\!\!{=}\!\!\!& \mathrm{diag}\left(\omega_{i,y}\right) 
\end{array}
\right. \!\!\! , \;\;\,
\omega_{i,x}, \omega_{i,y} \in \R_{++}^n ,  
\;\; i = 1,2 \, .
\label{eq:Ws}
\end{equation}
Note that the squared $W$-norms in (\ref{eq:ourRc})-(\ref{eq:ourRt}) are defined, for vectors $z = (z_x^\top, z_y^\top)^\top {\in}\, \mathbb{R}^{2n}$, as 
\begin{equation}
\|z\|_W^2 := z^\top W z = z_x^\top W_x z_x + z_y^\top W_y z_y \, .
\end{equation}

Notably, in close analogy to the TV semi-norm and the $G$-norm in Meyer's model \eqref{eq:Meyer_model}, the $H^1(W_1)$ and $H^{-1}(W_2)$ norms follow the structure of a dual pair, and their squared forms in \eqref{eq:ourRc} and \eqref{eq:ourRt} are Fenchel conjugates in the sense of weighted Hilbert spaces. This framework allows, independently of the specific choice of the weights, to effectively capture the antagonistic relationship between structural gradients and oscillatory patterns. 
While a strict duality is recovered for $W_2 = W_1^{-1}$, here we treat $W_1$ and $W_2$ as independent design parameters to provide greater flexibility in local adaptation. In particular, the positive weights in \eqref{eq:ourRc} and \eqref{eq:ourRt} allow for a spatially-adaptive, pixel-wise penalization of the cartoon gradients and the texture's underlying oscillatory components, respectively.

For the cartoon component $c$, the weighted quadratic regularizer in \eqref{eq:ourRc} extends beyond simple Tikhonov smoothing. In fact, as established in the Iteratively Reweighted Least Squares (IRLS) framework \cite{daubechies2010, chartrand2008}, an appropriate weight configuration enables the regularizer to promote derivative sparsity with an efficacy comparable to non-smooth and even non-convex functionals, such as truncated penalties that bridge unbiased sparse recovery and differentiable optimization \cite{Yang2026TruncatedHuber}. Specifically, by setting the weights to be inversely proportional to the derivative magnitudes, the weighted $H^1$ semi-norm can effectively emulate $L_p$ (with $p \leq 1$) sparsity-promoting behaviors, thereby enabling a more effective enforcement of gradient sparsity than the standard TV approach, which is limited by its convex $L_1$ nature.

Parallel arguments justify the adoption of the squared weighted $H^{-1}$ norm for the texture component $t$. The use of $H^{-1}$ spaces to model oscillatory signals was pioneered by Vese and Osher \cite{LeVese1} and Aujol et al. \cite{aujol2005}, who established this norm as a robust approximation of Meyer’s $G$-norm. 
Our framework, instead, focuses on a squared and weighted $H^{-1}$ formulation. This choice, rooted in the IRLS theory \cite{daubechies2010}, allows the quadratic penalty $\|\xi\|_{W_2}^2$ to act not only as a highly flexible approximation for $L_p$ minimization, including $L_\infty$ inherent in Meyer’s $G$-norm, but also as a spatially-adaptive masking tool. By decoupling the weights from a rigid functional regime, they can serve as localized design parameters capable of encoding prior structural knowledge, such as edge orientation or texture density, directly into the variational framework.

In order to define the proposed model \eqref{eq:our0}--\eqref{eq:ourRt} as explicitly as possible, we first introduce the definitions of the discrete gradient operator $\nabla$ (acting on the vectorized cartoon component $c \in \mathbb{R}^n$) and the discrete divergence operator $\mathrm{div}$ (acting on the vectorized vector field $\xi = [\xi_x^\top, \xi_y^\top]^\top \in \mathbb{R}^{2n}$):
\begin{equation}
\nabla c =  
\left[ \!
\begin{array}{c}
\nabla_x c \\
\nabla_y c
\end{array}
\! \right] , 
\quad
\mathrm{div}(\xi) = -\nabla^\top \xi = - (\nabla_x^\top \xi_x + \nabla_y^\top \xi_y),
\label{eq:Ddiv}
\end{equation}
where $\nabla_x, \nabla_y \in \mathbb{R}^{n \times n}$ are matrices representing the unscaled forward finite difference operators, approximating the horizontal and vertical partial derivatives of an $h \times w$ image ($n=hw$), respectively. By assuming anti-reflective boundary conditions, these matrices are defined as:
\begin{equation}
\nabla_x = I_h \otimes B_w, \quad \nabla_y = B_h \otimes I_w, \quad
B_m = \begin{bmatrix} 
-1 & 1 & 0 & \dots & 0 \\
0 & -1 & 1 & \dots & 0 \\
\vdots & \vdots & \ddots & \ddots & \vdots \\
0 & 0 & \dots & -1 & 1 \\
0 & 0 & \dots & -1 & 1 
\end{bmatrix} \in \mathbb{R}^{m \times m} ,
\label{eq:Dxy}
\end{equation}   
where $I_m$ denotes the identity matrix of order $m$ and $\otimes$ the Kronecker product.
Then, by plugging \eqref{eq:ourRc} and \eqref{eq:ourRt} into \eqref{eq:our0} and replacing $t = -\nabla^\top \xi$ in the data fidelity term, we obtain the explicit, unconstrained version of the proposed quadratic decomposition model:
\begin{equation}
\{\widehat{c},\widehat{\xi}\} \in \arg \min_{c,\xi}
\left\{ \,
\frac{1}{2}\| c-\nabla^\top \xi - f \|_2^2
+ \frac{\lambda_1}{2} \| \nabla c \|^2_{W_1}
+ \frac{\lambda_2}{2} \| \xi \|^2_{W_2}
\right\}, 
\label{eq:our}
\end{equation}
where $\lambda_1, \lambda_2 \in \R_{++}$ are scalar regularization parameters. After solving \eqref{eq:our}, the texture component is then immediately reconstructed by
\begin{equation}
\widehat{t} \,\;{=}\;\, -\nabla^\top \widehat{\xi} \, .
\end{equation}

We would like to highlight how the separation of the scalar regularization parameters $\lambda_i$ from the $W_i$ matrices is a modeling choice, but well motivated. In fact, as it will be shown in Proposition \ref{prop:MAP}, this separation emerges naturally from the probabilistic MAP formulation of the decomposition inverse problem: $\lambda_i$ represent global noise-to-signal variance ratios and serve as global fidelity/regularization balancing coefficients, while $W_i$ encode local pixel-wise properties (covariances) of the sought components.

Since $W_1$ and $W_2$ are positive definite, the objective function is strictly convex and the decomposition is well-posed.
Accurate separation of smoothing and edge-preserving behavior via spatially varying weights is central to high-quality cartoon--texture decomposition but is also intrinsically challenging.   
The ideal weights $W_1$ and $W_2$ should promote a piecewise smooth component $c$ (cartoon) and a highly oscillatory component $t$ (texture), i.e., regions with strong edges are regularized differently from flat or textured regions, thereby enhancing the decomposition quality. The per-pixel adaptivity provides nontrivial flexibility: the model remains quadratic but adjusts to local image features.

To analyze the existence and uniqueness of the solution to the proposed variational model \eqref{eq:our}, we define the joint solution vector $x \in \mathbb{R}^{3n}$ and the block matrices $S \in \mathbb{R}^{n \times 3n}$, $G \in \mathbb{R}^{2n \times 3n}$, and $R \in \mathbb{R}^{2n \times 3n}$ as follows:
\begin{equation}
x = \left[\!
\begin{array}{c}
 c \\
 \xi
\end{array}
\!\right] , 
\;\;\;\:
S = \big[ \, I_n \: -\nabla^\top \, \big] ,
\;\;\;\:
G = \big[ \, \nabla \;\: 0_{2n{\times}2n} \, \big] ,
\;\;\;\:
R = \big[ \, 0_{2n\,{\times}\,n} \; I_{2n} \, \big] ,
\label{eq:defs}
\end{equation}
where $0_{h{\times}w}$ denotes the $h \times w$ zero matrix. 
Based on the definitions in \eqref{eq:defs}, we have that
$S x = c -\nabla^\top \xi$, $G x = \nabla c$, $R x = \xi$, hence model \eqref{eq:our} can be equivalently reformulated as
\begin{equation}
\label{General_obj}
\widehat{x } \,\;{=}\; \arg \min_{x \in \R^{3n}} 	
	\left\{ \, \mathcal{J}(x) \;{:=}\; \frac{1}{2}\|S x-f\|_2^2
	+\frac{\lambda_1}{2}\|Gx\|_{W_1}^2
	+\frac{\lambda_2}{2}\|Rx\|_{W_2}^2
    \right\} \, .
\end{equation}
The following result establishes the existence and uniqueness of the solution.

\begin{proposition}
\label{prop1}
For any $f \in \R^n$ and any pair of regularization parameters $\lambda_1,\lambda_2 \in \R_{++}$ and of weight matrices $W_1, W_2 \in \R_{++}^{2n \times 2n}$ defined as in \eqref{eq:Ws}, the quadratic cost function $\mathcal{J}$ in \eqref{General_obj} is strictly convex. 
Hence, $\mathcal{J}$ admits a unique global minimizer given by the solution of the symmetric positive definite linear system yielded by the first-order optimality conditions:
\begin{equation}
\label{eq:normalA_main}
\begin{array}{c}
\nabla_x \:\! \mathcal{J}(x) = 0_{3n} 
\:\;\;{\Longleftrightarrow}\;\:\:
A(W_1,W_2)\,x \;=\: S^\top \! f \, , 
\vspace{0.3cm}\\
\text{with}\;\;\; 
A(W_1,W_2)\:=\: S^\top S + \lambda_1 G^\top W_1 G + \lambda_2 R^\top W_2 R 
\;\, \in \R^{3n \times 3n}
\end{array}
\end{equation}
\end{proposition}

\begin{proof}
The first-order optimality conditions for the quadratic function $\mathcal{J}$ in \eqref{General_obj} read 
\begin{equation}
\nabla_x \:\! \mathcal{J}(x) = 0_{3n} 
\:\;\;{\Longleftrightarrow}\;\:\:
S^\top(Sx - f) + \lambda_1 G^\top W_1 G\, x 
+ \lambda_2 R^\top W_2 R\, x = 0_{3n} \, .
\nonumber
\end{equation}
Rearranging yields the normal equations
\[
  \bigl(\underbrace{S^\top S + \lambda_1 G^\top W_1 G 
  + \lambda_2 R^\top W_2 R}_{=:\,A(W_1,W_2)}\bigr)\, x 
  = S^\top f,
\]
which is precisely the linear system in ~\eqref{eq:normalA_main}.

The matrix $A(W_1, W_2)$ is clearly symmetric and, based on the definitions in \eqref{eq:defs}, admits the block representation
\[
A(W_1, W_2) =
\begin{bmatrix}
I_n + \lambda_1\nabla^\top W_1 \nabla & -\nabla^\top \\
- \nabla & \nabla \,\nabla^\top  + \lambda_2 W_2
\end{bmatrix}.
\]
Since $\lambda_1, \lambda_2 \in \R_{++}$ and $W_1, W_2$ are diagonal and positive definite, both the diagonal blocks are symmetric positive definite. 
The Schur complement of the top-left block is given by
$$\mathcal{S}_c = (\nabla \nabla^\top + \lambda_2 W_2) - \nabla \, (I_n + \lambda_1 \nabla^\top W_1 \nabla)^{-1} \nabla^\top.$$
\noindent Using the Woodbury identity, we can express the inverse as 
\[(I_n + \lambda_1 \nabla^\top W_1 \nabla)^{-1} = I_n - \lambda_1 \nabla^\top (W_1^{-1} + \lambda_1 \nabla \nabla^\top)^{-1} \nabla.
\]
Since $\,W_1^{-1} + \lambda_1 \nabla \nabla^\top$ is positive definite and, thus, its inverse is positive definite, we have
\[
\mathcal{S}_c =\lambda_2 W_2 + \nabla \bigl[ \lambda_1 \nabla^\top (W_1^{-1} + \lambda_1 \nabla \nabla^\top)^{-1} \nabla \bigr] \nabla^\top \succeq \lambda_2 W_2 \succ 0 \, .
\]
\noindent Thus, according to the Schur Complement condition for positive definiteness of block matrices, the entire matrix $A(W_1, W_2)$ is symmetric positive definite, hence invertible. 
Therefore, the quadratic function in \eqref{General_obj} is strictly convex and admits a unique global minimizer obtained as the solution of the linear system in \eqref{eq:normalA_main}.
\end{proof}

The linear system in \eqref{eq:normalA_main} that solves the proposed quadratic decomposition model \eqref{General_obj} is of large (if not huge) size, and the coefficient matrix is symmetric, positive definite and sparse. Hence, it is efficiently solved using the iterative conjugate gradient (CG) method. The iterations are terminated as soon as the residual norm falls below a prescribed tolerance or a maximum number of iterations is reached. Solving the full coupled system ensures global consistency between \( c \) and \( \xi \), and is typically more efficient than alternating minimization schemes, which may require more iterations and can suffer from slower convergence due to partial updates.

Given the large number of free parameters in the proposed weighted variational model \eqref{eq:our}, an effective parameter selection strategy is essential to ensure high-quality decomposition. In the following Sections \ref{sec:PGVD} and \ref{sec:NGVD} we illustrate a probabilistic and a data-driven approach, respectively.

\section{Probabilistic-Guided Variational Decomposition (PGVD) framework}
\label{sec:PGVD}

With the aim of introducing a first, model-based statistical approach to the parameter estimation, we adopt a probabilistic framework that interprets the proposed variational formulation \eqref{eq:our} as deductively deriving from a Maximum A Posteriori (MAP) estimate of the sought components under precise assumptions on their probability distribution.

For this purpose, we introduce the variable $r := f - (c+t) \in \R^n$, which represents the residual of the texture/cartoon decomposition: it can be regarded as an additive noise (or error) affecting the observed image $f$.

First, we assume that the three components $r$, $c$, $\xi$ are mutually independent,
\begin{equation}
p(r,c,\xi) \,\;{=}\;\, p(r) \:\! p(c) \:\! p(\xi)  
\;\;\;{\Longrightarrow}\;\;\; 
p(r,c,t) \,\;{=}\;\, p(r) \:\! p(c) \:\! p(t)  \, ,
\label{eq:H1}
\tag{H.1}
\end{equation}
where the mutual independence of $r$, $c$, $t$ follows from $t$ being a deterministic function of $\,\xi$ via $t = \mathrm{div}(\xi)$.
Assuming that the additive noise that corrupts an image is independent of the image itself (or, in our case, of the two constituent components $c$, $t$) is quite standard. On the other hand, assuming that the cartoon and texture components are mutually independent seems reasonable, however it should be considered a modeling hypothesis and simplification. In fact, in correspondence with some edges in the image (in particular, related to physical edges of the objects in the scene), it is plausible that some correlation may exist. However, this correlation is spatially localized and should not fundamentally invalidate the global independence assumption for the purpose of deriving the variational model.

Then, we assume that the noise $r$ is the realization from a $n$-variate Gaussian-distributed random vector with zero-mean and scalar covariance matrix (so the noise is independent and identically distributed, i.i.d), in formula
\begin{equation}
p(r \mid \sigma_r) \,\;{=}\;\, \mathcal{G}_n\left(r;0_n,\Sigma_r\right)\,, \;\; \text{with}\;\; \Sigma_r = \sigma^2_r \:\! \mathrm{I}_n \, ,
\label{eq:H2}
\tag{H.2}
\end{equation}
where $\mathcal{G}_n(r; 0_n,\Sigma_r)$ indicates the value of the $n$-variate Gaussian distribution with mean vector $0_n$ and covariance matrix $\Sigma_r \in \R^{n \times n}$ evaluated at $r \in \R^n$, and with $\sigma_r \in \R_{++}$ denoting the noise standard deviation.

For what regards the cartoon component $c$, we assume it is the realization of a Non-Stationary Gaussian Markov Random Field (NS-GMRF), which, for $c$ 
regarded as a matrix with the generic entry indicated by $c_{i,j}$, reads
\begin{equation}
p(c) = \frac{1}{\mathcal{Z}_c} \prod_{i,j} \exp \left( -\frac{1}{2} \left[ \frac{(c_{i+1,j} - c_{i,j})^2}{\sigma_{x,i,j}^2} + \frac{(c_{i,j+1} - c_{i,j})^2}{\sigma_{y,i,j}^2} \right] \right) \, ,
\label{eq:NS_GMRF1}
\end{equation}
where $Z_c$ is the so-called partition function that ensures the normalization of the distribution. 
By noting that the differences in \eqref{eq:NS_GMRF1} represent unscaled forward finite difference approximations of the horizontal and vertical partial derivatives of $c$, thus corresponding to the definition of discrete gradient introduced in \eqref{eq:Ddiv}--\eqref{eq:Dxy}, the distribution in \eqref{eq:NS_GMRF1} can be equivalently written in the following compact form for the vectorized cartoon component $c$: 
\begin{align}
p(c \mid \Sigma_c) 
&\;{=}\;
\frac{1}{\mathcal{Z}_c} 
\prod_{i=1}^n \exp\left( - \frac{1}{2} \:\! \left[ 
\frac{(\nabla_x c)_i^2}{ \sigma_{x,i}^2} +  
\frac{(\nabla_y c)_i^2}{ \sigma_{y,i}^2}
\right] \right)
\nonumber
\\
&\;{=}\; 
\frac{1}{\mathcal{Z}_c} 
\exp \left( - \frac{1}{2} \:\! \| \nabla c \|_{{\Sigma_c}^{\!\!-1}}^2 \right) \, , \;\;\text{with}\;\; \Sigma_c = \mathrm{diag}(\Sigma_{c,x},\Sigma_{c,y}) \, ,
\label{eq:H3}
\tag{H.3}
\end{align}

Finally, we assume that the component $\xi$ follows an $2n$-variate Gaussian distribution with zero-mean and diagonal covariance matrix, in formula
\begin{equation}
p(\xi \mid \Sigma_\xi) \,\;{=}\;\, \mathcal{G}_{2n}\left(0_{2n},\Sigma_\xi\right)\,, \;\; \text{with}\;\; \Sigma_\xi = \mathrm{diag}\left(\Sigma_{\xi,x},\Sigma_{\xi,y}\right) \, . 
\label{eq:H4_1}
\tag{H.4}
\end{equation}
Based on noting that $t = \mathrm{div}(\xi)$ is defined via a linear transformation of a Gaussian-distributed random vector, it is easy to prove that assumption \eqref{eq:H4_1} for $\xi$ implies that $t$ is also Gaussian-distributed with zero-mean but non-diagonal covariance matrix; in particular, we have
\begin{equation}
p(t \mid \Sigma_t) \,\;{=}\;\, \mathcal{G}_{n}\left(0_{n},\Sigma_t\right)\,, \;\;\text{with}\;\; \Sigma_t = \nabla^\top \Sigma_\xi \nabla. 
\label{eq:H4_2}
\end{equation}

The covariance matrix $\Sigma_t$ in \eqref{eq:H4_2} serves as a sophisticated differential operator that inherently penalizes low-frequency components while promoting oscillatory patterns.

In the following Proposition~\ref{prop:MAP} we demonstrate - the proof is deferred to Appendix~\ref{app:proofs} - that starting from the probabilistic assumptions above and applying the MAP estimation approach, one obtains the proposed variational decomposition model \eqref{eq:our}.

\begin{proposition}
\label{prop:MAP}
Under assumptions \ref{eq:H1}, \ref{eq:H2}, \ref{eq:H3}, \ref{eq:H4_1},  the MAP estimation of components $c$ and $\xi$ leads uniquely to the proposed variational decomposition model \eqref{eq:our}, with the regularization parameters $\lambda_1,\lambda_2 \in \R_{++}$ defined as
\begin{equation}
\lambda_1 = \frac{\sigma_r^2}{\underline{\sigma}_{\,c}^2} \, ,
\quad\;
\lambda_2 = \frac{\sigma_r^2}{\underline{\sigma}_{\,\xi}^2} \, ,
\quad\;
\underline{\sigma}_{\, c}^2 = \min_{i=1,\ldots,2n} \Sigma_{c,ii} \, , 
\quad
\underline{\sigma}_{\,\xi}^2 = \min_{i=1,\ldots,2n} \Sigma_{\xi,ii} \, ,
\label{eq:MAPpars}
\end{equation}
and with the weight matrices $W_1,W_2 \in \R^{2n \times 2n}$ given by
\begin{equation}
W_1 = \underline{\Sigma}_{\,c}^{-1} \, ,
\quad\;
W_2 = \underline{\Sigma}_{\,\xi}^{-1} \, ,
\quad\;
\underline{\Sigma}_{\, c} = \frac{1}{\underline{\sigma}_{\, c}^2} \, \Sigma_{ c} \, ,
\quad\;
\underline{\Sigma}_{\,\xi} = \frac{1}{\underline{\sigma}_{\,\xi}^2} \Sigma_{\xi} \, .
\label{eq:MAPsigm}
\end{equation}
\end{proposition}

The distribution hyperparameters $\sigma_r^2$, $\Sigma_c$, $\Sigma_\xi$ - and, then, the model parameters $\lambda_1,\lambda_2$ and $W_1,W_2$ via (\ref{eq:MAPpars})-(\ref{eq:MAPsigm}) - can be estimated using a local Maximum Likelihood (ML) strategy embedded into an iterative solver to the model, adapted from \cite{SIAMrev2021,10.1007/978-3-319-68195-5_17}. In our case, we extend this framework from weighted Total Variation (TV)-like semi-norms to squared weighted $H$-norms.

Not only to simplify the estimation but also to make the estimated parameters more effective in terms of the quality of the corresponding decomposition, we adopt the following simplifications: 
(i) \( \Sigma_{c,x} = \Sigma_{c,y} \) and $\Sigma_{\xi,x} = \Sigma_{\xi,y}$, so that only two diagonal matrices \( \Sigma_c = \diag(\sigma_{c,1}^2,\ldots,\sigma_{c,n}^2) \) and $\Sigma_\xi = \diag(\sigma_{\xi,1}^2,\ldots,\sigma_{\xi,n}^2)$ need to be estimated; (ii) instead of estimating $\sigma_r$ and then setting $\lambda_1,\lambda_2$ via (\ref{eq:MAPpars}), $\lambda_1,\lambda_2$ are fixed in advance, hence only $W_1,W_2$ need to be estimated (\( \sigma_r^2 \) does not require estimation).

The basic idea of the estimation approach is that since the two regularization terms in \eqref{eq:our} come deductively from precise assumptions on the distribution of $c$ and $\xi$, then the pixel-based weights can be inferred by ML estimation of the hyperparameters that characterize the pixel-wise distribution.

To illustrate the pixel-wise estimation procedure of the target variances $\sigma_{c,i}^2$, $i = 1,\ldots,n$, we focus on a generic pixel and denote by $\sigma^2$ the target variance.
Then, we consider a square symmetric neighborhood of the pixel of radius N pixels - that is, a $(2 \mathrm{N}+1) \times (2 \mathrm{N}+1)$ neighborhood - and define the sample set for the estimation as the set of values of the considered variable, that we denote by $v$, in the neighborhood,
\begin{equation}
\mathcal{S} := \left\{ v_1, \ldots, v_Z \right\} \, , \quad \text{with}\;\; Z = (2\mathrm{N}+1)^2 \, .
\end{equation}
The samples in $\mathcal{S}$ are regarded as $Z$ independent realizations from the same distribution; in particular, based on the assumption \eqref{eq:H2} on the distribution of $c$, which can be regarded as  assuming a zero-mean Gaussian distribution with variance $\sigma^2_{c,i}$ for the gradient norm $\| (\nabla c)_i \|_2$ at each pixel, the negative log-likelihood of $\mathcal{S}$ reads
\begin{eqnarray}
-\ln p \left(\mathcal{S} \mid \sigma^2\right) 
& {=} & 
\frac{Z}{2} \ln (2 \pi) + \frac{Z}{2} \ln \sigma^2 
+\frac{1}{2 \sigma^2} \sum_{j=1}^Z v_j^2 \, .
\end{eqnarray}
It follows that the maximum likelihood (or, equivalently, the minimum negative log-likelihood) estimate $\widehat{\sigma}^2$ of the variance $\sigma^2$ is simply given by
\begin{equation}
\widehat{\sigma}^2 = \arg \min_{\sigma^2} \left\{ 
\frac{Z}{2} \ln \sigma^2 
+\frac{1}{2 \sigma^2} \sum_{j=1}^Z v_j^2 \right\}
=
\frac{1}{Z} \sum_{j=1}^Z v_j^2 .
\label{eq:rrr}
\end{equation}
 
Using the pixel-wise estimation formula above for all pixels, we can easily compute an estimate of the total diagonal covariance matrix $\Sigma_c$, reading
\begin{equation}
\widehat{\Sigma}_c = \diag\left(\widehat{\sigma}_{c,1}^2,\ldots,\widehat{\sigma}_{c,n}^2\right) \, ,
\end{equation}
then, we compute
\begin{equation}
\widehat{\underline{\sigma}}_{\,c}^2 := \min_{i=1,\ldots,n} \widehat{\sigma}_{c,i}^2
\;\;{\Longrightarrow}\;\;
\widehat{\underline{\Sigma}}_{\,c} = \frac{1}{\widehat{\underline{\sigma}}_{\,c}^2} \widehat{\Sigma}_c \, .
\end{equation}
The above procedure can be similarly carried out to estimate the diagonal covariance matrix $\widehat{\Sigma_{\xi}}$ of the texture components.

Finally, the parameters of the model (regularization parameters $\lambda_1$, $\lambda_2$ and the weight matrices $W_1$, $W_2$) are fixed/estimated based on the formulas in \eqref{eq:MAPpars}--\eqref{eq:MAPsigm}.
In particular, in accordance with \eqref{eq:rrr}, the weight $\omega_{c,q}$ associated to the $q$-th pixel location in the vectorized image $c$ - corresponding to the pixel location $(i,j)$ in the original image - is computed by
\begin{equation}
\widehat{\omega}_{c,q} = \biggl(\epsilon+\frac{1}{2 Z}\sum_{(l,m)\in \mathcal{N}_{i,j}^{\mathrm{N}}}\left\| (\nabla c)_{l,m}\right\|_2^2\biggr)^{-1},
\label{eq:PRMET}
\end{equation}
where $\mathcal{N}_{i,j}^{\mathrm{N}}$ indicates the square neighborhood of radius $\mathrm{N}$ pixels, and the fixed parameter $\epsilon>0$ prevents division by zero.

The approach outlined above relies on knowing the two sought components $c$ and $\xi$, which is clearly not the case. Therefore, we propose an iterative procedure.
Starting with $c^{(0)} = f$ and $W_1^{(0)}=W_2^{(0)}=\mathrm{I}_{2n}$, the weight matrices are updated, according to \eqref{eq:PRMET}, into $W_1^{(k+1)}$ and $W_2^{(k+1)}$, and, then, the decomposition components are updated, by solving the quadratic optimization problem \eqref{eq:our} - that is, by solving the linear system in \eqref{eq:normalA_main} - into $c^{(k+1)}$ and $\xi^{(k+1)}$.

\vspace{0.5cm}

This probabilistic approach for identifying the weight parameters, and consequently the decomposition components, can produce high-quality results, as it will be shown in the experimental section. However, it operates under a single-instance paradigm, where the model relies solely on a single observed image \( f \) to infer its constituent components \( c \) and \( t \).  

When multiple labeled training pairs are available, the model can benefit from a supervised multi-instance framework. Unlike the single-instance method, which must rely entirely on the structural information in a single image, the supervised setting enables learning to generalize across diverse examples. This added information allows the model to predict spatially adaptive weights more accurately, potentially leading to improved decompositions.  
These observations motivate the neural-guided variational decomposition framework developed in the next section.

\section{Neural-Guided Variational Decomposition (NGVD) framework}
\label{sec:NGVD}
The selection of optimal model parameters is formulated as a multi-instance supervised learning framework, where we assume access to $M$ training samples 
\begin{equation}
\{(f^{(i)}, g^{(i)})\}_{i=1}^M \, , \quad \text{with} \;\; 
g^{(i)} := (c^{(i)}, t^{(i)}) \, ,
\label{eq:dataset}
\end{equation}
denoting the desired cartoon and texture components. To enable data-driven decomposition, we introduce two parameterized prediction maps:
\begin{equation}
\lambda = \Lambda_{\Theta_1}(f) \,, \qquad
W = \mathcal{W}_{\Theta_2}(x) \,,
\label{eq:WL}
\end{equation}
where $\Lambda_{\Theta_1}$ is a scalar multilayer perceptron (MLP) that outputs the regularization parameters $\lambda := [\lambda_1,\lambda_2]$, and $\mathcal{W}_{\Theta_2}$ is a convolutional U-Net that predicts spatially adaptive weights $W := [W_1, W_2]$. The full network is parameterized by $\Theta = (\Theta_1, \Theta_2)$.

The identification of the optimal parameters $\widehat{\Theta}$ of the two networks, useful at prediction time to determine the variational model parameters $(\lambda_1, \lambda_2,W_1,W_2)$ is formulated as the solution of the following bilevel optimization problem:
\begin{subequations}
\label{eq:NGVD_bilevel}
\begin{align}
&\widehat{\Theta} \in \arg \min_{\Theta} \;\;
 \frac{1}{2M} \sum_{i=1}^M 
\left\|
\mathcal{D}\widehat{x}^{(i)}(\Theta) - g^{(i)}
\right\|_2^2 \quad \text{s.t.}
\label{eq:NGVD_upper}
\\[4pt]
& \widehat{x}^{(i)}(\Theta)
= \argmin_{x \in \mathbb{R}^{n+2n}}
\left\{
\frac{1}{2}\|Sx - f^{(i)}\|_2^2
+ \frac{\lambda_1^{(i)}}{2}\|Gx\|_{W_1^{(i)}}^2
+ \frac{\lambda_2^{(i)}}{2}\|Rx\|_{W_2^{(i)}}^2
\right\}\!,
\label{eq:NGVD_lower}
\\[2pt]
& \lambda^{(i)} = \Lambda_{\Theta_1}\!\left(f^{(i)}\right), 
\;\;
W^{(i)} = \mathcal{W}_{\Theta_2}\!\left(\widehat{x}^{(i)}\right),
\;\; i=1,\ldots,M,
\label{eq:NGVD_maps}
\end{align}
\end{subequations}

with $g^{(i)}$ defined in \eqref{eq:dataset}, $\mathcal{D}$ represents the block diagonal operator \(\mathcal{D} : \mathbb{R}^{n+2n} \to \mathbb{R}^{2n}\), acting on \(\hat{x} = (c, \xi)\), as:
\[
\mathcal{D} := 
\begin{bmatrix}
I & 0 \\
0 & \operatorname{div}
\end{bmatrix}.
\]
The upper-level loss function represents the Mean Square Error (MSE) metrics of goodness of the estimated parameters $\Theta$, and the lower-level minimization problem aims at computing the solution components, giving two fixed $\Theta$-parametrized maps. 

The proposed training procedure, detailed in Algorithm \ref{alg:train}, follows the above bilevel optimization paradigm. This approach bridges the gap between classical variational methods and deep learning by embedding the GVD physical model  within a supervised learning framework.

The procedure begins by passing the input image $f^{(i)}$ through a parameter predictor $\Lambda_{\Theta_1}$. Unlike traditional variational methods that rely on manually tuned hyperparameters, this neural-guided component learns to map image features to optimal regularization parameters $(\lambda_1^{(i)}, \lambda_2^{(i)})$. This ensures that the decomposition is tailored to the specific structural characteristics of each observation.

\begin{algorithm}[t]
\caption{Training the NGVD framework}
\label{alg:train}
\begin{algorithmic}[1]
\REQUIRE Dataset $\{(f^{(i)},g^{(i)})\}_{i=1}^M$
\ENSURE Weights $\Theta$ for operators $\Lambda_{\Theta_1}$ and $\mathcal{W}_{\Theta_2}$
\WHILE{not converged}
  \FOR{$i \gets 1$ \TO $M$}
    \STATE $(\lambda_1^{(i)},\lambda_2^{(i)}) \gets \Lambda_{\Theta_1}(f^{(i)})$ \COMMENT{regularization parameters}
    \STATE $\widehat{x}_0^{(i)} \gets (f^{(i)},0)$
    \FOR{$k \gets 1$ \TO $K$}
      \STATE $(W_1^{(i)},W_2^{(i)}) \gets \mathcal{W}_{\Theta_2}(\widehat{x}_{k-1}^{(i)})$
      \STATE Solve for $\widehat{x}_k^{(i)}$:
      \STATE \hspace{1.6em}$A(W_1^{(i)},W_2^{(i)})\,x = S^\top f^{(i)}$
    \ENDFOR
    \STATE Update loss with $\widehat{x}^{(i)}$ in \eqref{eq:NGVD_lower}
  \ENDFOR
  \STATE $\Theta \gets$ minimize the loss in \eqref{eq:NGVD_upper}
\ENDWHILE
\STATE \textbf{return} $\widehat{\Theta}$
\end{algorithmic}
\end{algorithm}

\begin{algorithm}[t]
\caption{Prediction by the NGVD framework}
\label{alg:test}
\begin{algorithmic}[1]
\REQUIRE Observation $f$, number of iterations $K$, $\Lambda_{\Theta_1}$, $\mathcal{W}_{\Theta_2}$
\ENSURE Decomposed components $\widehat{x}=(\widehat{c},\widehat{t})$
\STATE $(\lambda_1,\lambda_2) \gets \Lambda_{\Theta_1}(f)$ \COMMENT{regularization parameters}
\STATE $\widehat{x}_0 \gets (f,0)$
\FOR{$k \gets 1$ \TO $K$}
  \STATE $(W_1,W_2) \gets \mathcal{W}_{\Theta_2}(\widehat{x}_{k-1})$
  \STATE Solve for $\widehat{x}_k$:
  \STATE \hspace{1.6em}$A(W_1,W_2)\,x = S^\top f$
\ENDFOR
\STATE \textbf{return} $(\widehat{c},\widehat{t}) \gets \mathcal{D}\widehat{x}_K$
\end{algorithmic}
\end{algorithm}

The core of the algorithm is an inner loop of $K$ iterations for the weights refinement. In each step $k$, a second neural network, $\mathcal{W}_{\Theta_2}$, observes the current state of the decomposition $\widehat{x}_{k-1}$ to update the weighting operators $(W_1, W_2)$.
Rather than treating the decomposition as a ``black-box'' regression, the algorithm solves the linear system $A(W_1, W_2)x = S^\top f$ approximately via CG (at most $80$ iterations, tolerance $\epsilon=10^{-6}$), warm-started from $\widehat{x}_{k-1}$. This ensures that the output $\widehat{x}_k$ always satisfies the underlying variational principles of the cartoon-texture model, while the neural network guides the trajectory toward the ground truth.

By training on a dataset of $M$ labeled samples $\{ (f^{(i)}, g^{(i)}) \}_{i=1}^M$, the framework leverages the ``multi-instance'' advantage discussed previously. While the final model can operate in a single-instance mode (performing inference on one new image), the training phase uses the collective insights of the entire dataset. This allows the parameters $\Theta$ to generalize across various textures and geometries, leading to a more robust and ``insightful'' decomposition than could be achieved by optimizing a single image in isolation.

The obtained weights $\Theta$ allow for the construction of the prediction operators  ${\Lambda}_{\Theta_1}, \mathcal{W}_{\Theta_2}$ in \eqref{eq:WL}, which are then used in the inference process for the decomposition of an observed image $f$, as described in Algorithm \ref{alg:test}. The decomposition predictive algorithm clarifies the roles of model design and data-driven numerical optimization: the outer loop is responsible for producing reliable structural guidance (by network-based weight refinement), while the inner minimization exploits the simple quadratic form of the functional to compute accurate reconstructions given that guidance. 


To leverage the power of multi-instance supervised learning, we allowed both the weight-prediction mechanism and the estimate of the regularization parameters to be jointly optimized, enabling the model to adaptively learn decomposition strategies from data and thereby achieve superior performance in separating cartoon and texture components. 

Central to this embedding are two neural networks: a scalar MLP, denoted by $\Lambda_{\Theta_1}$, for predicting the regularization parameters $\lambda_1,\lambda_2$ in the variational model, and a U-Net, named  $\mathcal{W}_{\Theta_2}$, for generating spatially adaptive weight matrices $W_1,W_2$. 

The regularization parameters $\lambda_1, \lambda_2 \in \R_{++}$ 
in \eqref{General_obj}, are estimated once at the beginning from the input observation $f$ 
by $\Lambda_{\Theta_1}$ in \eqref{eq:WL}. 
This network outputs two positive scalars, ensuring positivity through a {\em softplus} activation function in the final layer, defined as
\begin{equation}
\label{eq:softplus}
\operatorname{softplus}(s) := \log(1 + e^s),
\end{equation}
which smoothly enforces strict positivity. 
The $\Lambda_{\Theta_1}$ first applies global average pooling over all spatial dimensions, reducing the input to a channel-wise feature vector of length~$C$,  followed by a fully connected layer with ReLU activation and a fully connected layer with softplus activation, as illustrated in Figure~\ref{fig:placeholder}(b).
Since the fully connected layers depend only on the channel 
dimension~$C$, the network is resolution-agnostic and can be applied at 
inference time to images of arbitrary size without retraining. This is consistent with the 
role of $\lambda_1$ and $\lambda_2$ as image-level global scalars: 
predicting them from a global image descriptor -- rather than from 
spatially varying features -- is both natural and sufficient, as 
confirmed by the ablation study in Table~\ref{tab:ablation_results}. Pixel-level spatial 
adaptivity is instead the responsibility of the weight-predicting operator $\mathcal{W}_{\Theta_2}$.
By predicting these parameters directly from the input, the network can tailor the global trade-offs to the specific characteristics of the observed image, without manual tuning.

The weight-predicting operator $\mathcal{W}_{\Theta_2}$ is implemented as a lightweight U-Net, which generates the diagonal entries of the two positive definite matrices $(W_1,W_2)$ at each iteration $k$. 
To provide rich contextual information, $\mathcal{W}_{\Theta_2}$ takes as input a concatenation of the reconstructed cartoon and texture component estimator
$$(W_1, W_2) = \mathcal{W}_{\Theta_2}( \widehat x_{k-1}).$$
The U-Net architecture, detailed in Figure~\ref{fig:placeholder}(c), features an encoder-decoder structure with convolutional layers, 
Leaky ReLU activations, max-pooling for downsampling, and upconvolutions for upsampling. 
The output layer employs a {\em sigmoid} activation function, defined as
\begin{equation}
\label{eq:sigmoid}
\operatorname{sigmoid}(s) := \frac{1}{1 + e^{-s}},
\end{equation}
to guarantee strictly positive weight maps, aligning with the requirements for convexity and stability.

\begin{figure}[tbp]
    \centering
    \includegraphics[width=1\linewidth]{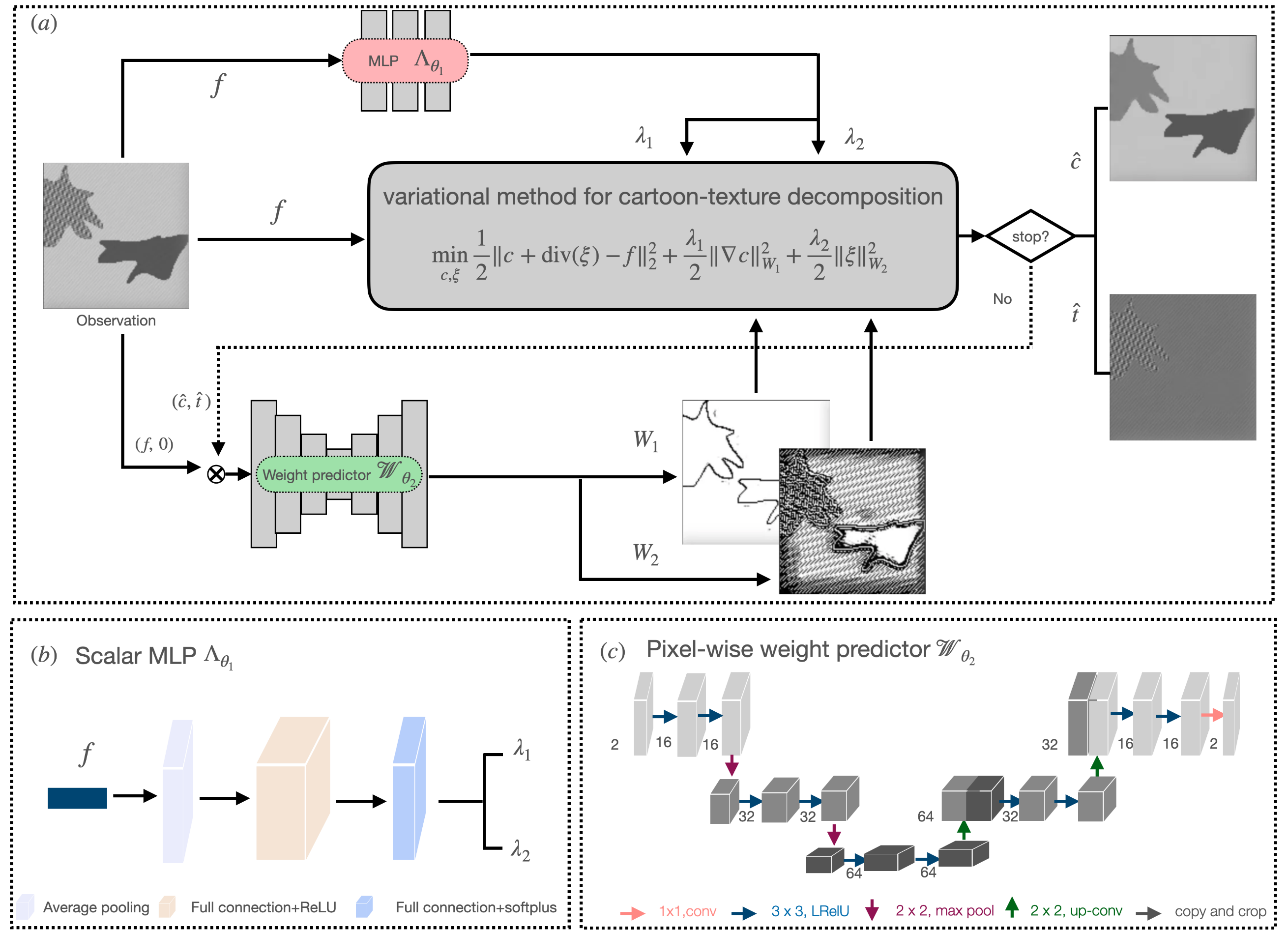}
    \caption{Overview of the proposed NGVD framework: (a) the guided variational network for image decomposition; (b) the scalar MLP predicting global regularization scalars \(\lambda_1,\lambda_2\); (c) the detailed structure of weight predictor \(\mathcal W_{\Theta_2}\), initialized with $\widehat{x}_0 = (f,0).$}
    \label{fig:placeholder}
\end{figure}

The complete framework, named \textbf{Neural Guided Variational Decomposition}, encapsulates the iteratively guided decomposition  procedure with these learned components:
$$(\widehat c, \widehat t ) = \widehat x_{K}  := \mathrm{NGVD}\big(f; \Lambda_{\Theta_1}, \mathcal{W}_{\Theta_2}\big),$$
where $\mathrm{NGVD}$ refers to the guided variational decomposition workflow outlined in  Algorithm~\ref{alg:test}, using the neural weight-prediction in Algorithm~\ref{alg:train}. 

Figure~\ref{fig:placeholder} provides an overview of the proposed NGVD framework, illustrating the guided variational decomposition pipeline (a), the scalar MLP for global regularization parameters (b), and the pixel-wise U-Net for weight prediction (c). This design bridges classical variational methods with deep learning, combining the interpretability and stability of optimization-based decompositions with the flexibility of data-driven adaptation.

\section{Convergence Analysis for NGVD}
\label{sec:theory}

In this section, we cast the proposed variational decomposition and the iteratively guided solver into a single fixed-point framework and state the main theoretical properties: (i) uniqueness of each fixed-weight iteration, (ii) existence of a fixed point, (iii) a sufficient, verifiable contractivity condition that guarantees convergence to a unique fixed point, and (iv) stability estimates with respect to perturbations in the observed data. This fixed-point viewpoint and strategy are inspired by \cite{lenzen2015solution,pourya2025dealing}. Note that \(\lambda_1\) and \(\lambda_2\) are predicted once from the input \(f\) and remain global and fixed throughout the iterations, while weight-predicting operator \(\mathcal{W}_{\Theta}\) (we slightly misuse the symbol $\Theta_2$ as $\Theta$ in this section) is shared across all iterations.

\subsection{Fixed-Point Reformulation and Notation}
For learned diagonal, positive definite weig-ht matrices \((W_1,W_2):=\mathcal{W}_{\Theta}(x)\) and fixed positive scalars \(\lambda_1,\lambda_2>0\), the minimizer of objective function \eqref{General_obj}, i.e.,
\[
J(x;\mathcal{W}_{\Theta})
=\frac{1}{2}\|Sx-f\|_2^2
+\frac{\lambda_1}{2}\|Gx\|_{W_1}^2
+\frac{\lambda_2}{2}\|Rx\|_{W_2}^2,
\]
can be expressed as
\begin{equation}\label{fixpointmap}
    \mathcal{T}_{\theta}(x) = \argmin_{x}\; J(x;\mathcal{W}_{\Theta}).
\end{equation}

The corresponding normal equation is
\begin{equation}
\label{eq:solve}
A(W_1,W_2)\, x \;=\; b,\qquad
A(W_1,W_2) :=  S^\top S + \lambda_1 G^\top W_1 G + \lambda_2 R^\top W_2 R,
\end{equation}
with \(b := S^\top f\). Then, the refinement scheme coincides with the Picard (fixed-point) iteration
\begin{equation}
x_{k} \;=\; \mathcal{T}_{\theta}(x_{k-1}) \;,
\label{eq:fp}
\end{equation}
where in \eqref{eq:solve} \((W_1,W_2)=\mathcal{W}_{\Theta}(x_{k-1})\) as the adaptive weights corresponding to \(\mathcal{T}_{\theta}(x_{k-1})\).
In every iteration $k$, we ``freeze'' the weights based on $x_{k-1}$, solve the now-linear system of normal equations and obtain the new $x_k$.

The process \eqref{eq:fp} can be interpreted as an infinite-depth neural network.
If $x_k \rightarrow \widehat{x}$, then   $\mathcal{T}_{\theta}(\widehat{x})=\widehat{x}$, and $\widehat{x}$ is a fixed point of the operator.
Theoretically, the network is trained to turn $\mathcal{T}_{\theta}$ into a contraction towards the desired solution.

\subsection{Admissible Weights and Computable Constants}

During all iterations, we restrict admissible diagonal weight matrices to satisfy uniform bounds
\[
\omega_{\min} I \preceq W_i \preceq \omega_{\max} I,\qquad i=1,2,
\]
for some constants \(0<\omega_{\min}\le\omega_{\max}<1\) (these are enforced in practice by final sigmoid activation \eqref{eq:sigmoid} + clipping of the U-Net outputs). Introduce the stacked operator
\begin{equation}\label{defineMforlowerbound}
    \mathcal{M} :=
\begin{bmatrix}
\,S\\[4pt]
\sqrt{\lambda_1\,\omega_{\min}}\,G\\[4pt]
\sqrt{\lambda_2\,\omega_{\min}}\,R
\end{bmatrix}.
\end{equation}
The associated Gramian matrix $\mathcal{M}^\top \mathcal{M}$ is full-rank. In fact, it is given by $\mathcal{M}^\top \mathcal{M} ~= A(\omega_{\min}I,\, \omega_{\min}I)$, which is exactly the system matrix of Proposition~3.1 evaluated at $W_1 = W_2 = \omega_{\min}I$, that is positive definite. It follows that \(\mathcal{M}\) is full column rank.
Define the lower-bound constant
\begin{equation}\label{singlerM}
    \alpha := \sigma_{\min}^2(\mathcal{M}).
\end{equation}
Here, \(\sigma_{\min}(\mathcal{M})\) is the smallest singular value of \(\mathcal{M}\) and $\sigma_{\min}(\mathcal{M})>0$ since the full column rank property of $\mathcal{M}$.    We denote operator norms \(\|S\|,\|G\|,\|R\|\) (spectral norms) and the Euclidean norm of the vectorized measurement \(\|f\|_2\).

\subsection{Bounding the Lipschitz Constant of $\mathcal{W}_{\Theta}$}
\label{subsec:lipschitz_bound}
To ensure the contractivity condition for convergence, we derive a computable upper bound on the Lipschitz constant $L_{\mathcal{W}}$ of $\mathcal{W}_{\Theta}$, using real spectral normalization (realSN \cite{ryu2019plug}) on its convolutional layers. RealSN extends spectral normalization \cite{miyato2018spectral} by directly computing the spectral norm of the convolutional operator via power iteration on tensor representations, without reshaping kernels into matrices. Specifically, for each convolutional layer with kernel $K_l$, realSN maintains singular vector estimates $U_l, V_l$ and performs power iterations: $V_l \leftarrow K_l^*(U_l) / \|K_l^*(U_l)\|_2$, $U_l \leftarrow K_l(V_l) / \|K_l(V_l)\|_2$, where $K_l^*$ is the adjoint convolution. The kernel is then normalized as $K_l \leftarrow c_l K_l / \sigma(K_l)$, with $\sigma(K_l) = \langle U_l, K_l(V_l) \rangle$, ensuring the layer's Lipschitz constant is at most $c_l$. This enables control over the network's overall Lipschitz constant during training, as detailed in \cite{ryu2019plug}. Although our U-Net includes max-pooling and bilinear upsampling, these operations have bounded Lipschitz constants (e.g., max-pooling and bilinear interpolation are 1-Lipschitz under the $\ell_2$-norm). Note that realSN is employed here solely for the purpose of theoretical convergence analysis and is not utilized in the actual training process; the weight bounds are instead enforced through activation functions and clipping in practice.
\begin{proposition}[Computable Lipschitz bound for $\mathcal{W}_{\Theta}$]\label{prop:LW_bound}
Assume $\mathcal{W}_{\Theta}$ is a U-Net with $N$ convolutional layers (each real spectrally normalized with any factor $c_i > 0$) followed by Leaky ReLU activations (1-Lipschitz). Then, the Lipschitz constant satisfies
$$L_{\mathcal{W}} \le \kappa := \prod_{i=1}^N c_i.$$
\end{proposition}
\begin{proof}
By spectral normalization, each convolutional layer $l$ has Lipschitz constant at most $c_l$, and each Leaky ReLU activation is $1$-Lipschitz. Since the composition of Lipschitz maps has Lipschitz constant bounded by the product of the individual constants \cite{miyato2018spectral,ryu2019plug}, the result follows.
\end{proof}

RealSN makes the per-layer bounds explicit and computable post-training. If $\kappa$ exceeds the desired value for contractivity, the factors $c_i$ can be adjusted to reduce it, or alternatively, renormalize U-Net outputs by a factor to scale $L_{\mathcal{W}}$ down, adjusting $\omega_{\min}, \omega_{\max}$ accordingly while maintaining the admissibility bounds. This bound, inspired by Lipschitz analyses, enables fully computable convergence criteria below.

\subsection{Theoretical Results}
\label{sec:tr}
In this subsection, we present the main theoretical results. Their complete proofs are given in Appendix~\ref{app:proofs}.

\begin{lemma}
\label{Alowerbound}
Let \(\mathcal M\) be defined as in \eqref{defineMforlowerbound}.
For any \(\lambda_1,\lambda_2>0\) and  weight matrices  $(W_1,W_2)$ satisfying the bound \(\omega_{\min} I \preceq W_i \preceq \omega_{\max} I, i=1,2\),
we have the following lower bound
\[
x^\top A(W_1,W_2) x \ge \| \mathcal M x\|_2^2 \ge \sigma_{\min}^2(\mathcal M)\,\|x\|_2^2,
\]
\noindent and upper bound
\begin{equation}
\|A(W_1,W_2)^{-1}\| \le \frac{1}{\alpha} \, ,
\label{eq:AinvL}
\end{equation}
where \(\alpha\) is given in \eqref{singlerM}.
\end{lemma}

\begin{theorem}[Lipschitz bound]
\label{thm:lipschitz_main}
Let \(\mathcal{W}_{\Theta}\) denote the weight operator mapping \(x\) to diagonal matrices $W_1$ and $W_2$. Let \(L_{\mathcal{W}}\) be the Lipschitz constant of   \(\mathcal{W}_{\Theta}\) on the ball \(\mathcal{B}=\{x:\|x\|_2\le r\}\), i.e.,
\[
\|\mathcal{W}_{\Theta}(x)-\mathcal{W}_{\Theta}(y)\|
\le L_{\mathcal{W}} \|x-y\|_2,\quad \forall x,y\in\mathcal{B},
\]

where $\|\mathcal{W}_\Theta(x)\| := \| \, \omega(x) \, \|_{\infty}$, with $\omega(x) = [ \, \mathrm{diag}(W_1(x)),\mathrm{diag}(W_2(x))\,] \in \R^{4n}$. 

Then, the mapping \(\mathcal{T}_{\theta}(x)\) is \(L_{\mathcal{T}}\) Lipschitz on \(\mathcal{B}\), i.e.
\[
\|\mathcal{T}_{\theta}(x)-\mathcal{T}_{\theta}(y)\|_2 \le L_{\mathcal{T}} \|x-y\|_2,
\]
with the explicit upper bound
\[
L_{\mathcal{T}} \le \frac{(\lambda_1\|G\|^2+\lambda_2\|R\|^2)\,L_{\mathcal{W}}\,\|S\|\,\|f\|_2}{\alpha^2},
\]
where \(\alpha\) is defined in \eqref{singlerM}. 
\end{theorem}

\begin{theorem}[Existence and contractive convergence]
\label{thm:exist_converge_main}
Let \(r := \|S\|\,\|f\|_2 / \alpha\). Then, \(\mathcal{T}_{\theta}(x)\) \eqref{fixpointmap} maps the closed ball \(\mathcal{B}=\{x:\|x\|_2\le r\}\) into itself. Moreover:
\begin{itemize}
  \item[(a)] (\emph{Existence}) \(\mathcal{T}_{\theta}(x)\) has at least one fixed point in \(\mathcal{B}\). 

  \item[(b)] (\emph{Uniqueness and convergence}) The computable quantity
$$\mathcal{Q} := \frac{(\lambda_1\|G\|^2+\lambda_2\|R\|^2)\,\kappa\,\|S\|\,\|f\|_2}{\alpha^2}$$
(with $\kappa$ from Proposition~\ref{prop:LW_bound}) satisfies $\mathcal{Q} < 1$ for the chosen set of normalization factors $c_i (i=1,\ldots,N)$. Under this condition, $\mathcal{T}_{\theta}(x)$ is a contraction on $\mathcal{B}$. In that case the iterates converge linearly to the unique fixed point \(x_\star\in\mathcal{B}\):
  \[
  \|x_{k}-x_\star\|_2 \le \mathcal{Q}^{k}\|x_0-x_\star\|_2.
  \]
\end{itemize}
\end{theorem}

In the following proposition we give an explicit stability bound.

\begin{proposition}[Stability to data perturbation] Let \(x^\star_f\) denote the unique fixed point of operator \(\mathcal{T}_{\theta}(x)\) corresponding to the observed image \(f\). Then, for fixed admissible weights \((W_1,W_2)\) one has the explicit stability bound:
\begin{equation}
\|x^\star_{f_1} - x^\star_{f_2}\|_2 \le \frac{\|S\|}{\alpha}\,\|f_1-f_2\|_2.
\label{eq:stabb}
\end{equation}
\label{prop:stab}
\end{proposition} 
\begin{proof}
For fixed admissible weights \((W_1,W_2)\), the two fixed points $x^\star_{f_1}$ and $x^\star_{f_2}$ of operator  $\mathcal{T}_{\theta}(x)$ associated with two different observations $f_1$ and $f_2$ are clearly both solution of normal equations \eqref{eq:normalA_main} with $f = f_1$ or $f = f_2$, respectively. Subtracting, we obtain:
\[
x^\star_{f_1}-x^\star_{f_2} = A(W)^{-1} S^\top (f_1-f_2).
\]
Finally, taking norms we easily get the stability bound in \eqref{eq:stabb}.
\end{proof}

Thus the reconstruction is Lipschitz stable with constant \(\|S\|/\alpha\) with respect to perturbations in the measurements, reinforcing the robustness of the fixed-point formulation.

\section{Numerical Experiments}
\label{sec:experiments}
In this section, we present numerical experiments to demonstrate the effectiveness of our proposed method for image decomposition. We utilize a combination of synthetic and real-world datasets, enabling quantitative evaluations under controlled conditions and qualitative assessments of practical applicability.

\subsection{Experimental Setup}
\label{sec:exp_setup}
For the synthetic dataset, we adopt the generation procedure outlined in \cite{doi:10.1137/24M1677770}. Each synthetic observation $f$ is constructed as $f = c + t$, where the cartoon component $c$ comprises piecewise constant regions or smooth gradients, and the texture component $t$ is generated using periodic patterns or stochastic processes. This yields a dataset of 512 training images, each sized $128 \times 128$ pixels, accompanied by ground-truth decompositions for precise metric computation. In addition, several real-world natural images are employed for visual evaluation of the method's robustness in practical scenarios.

Our approach is compared against traditional baselines and several state-of-the-art methods: the total-variation and G-norm (TV-Gnorm) model \cite{wen2019primal}, the low-rank and weighted least-squares (LR-WLS) method \cite{li2025cartoon}, the deep unfolding Low Patch Rank network (LPR-Net) \cite{10.1007/978-3-031-92366-1_30}, and the plug-and-play joint structure-texture (Joint-PnP) scheme \cite{doi:10.1137/24M1677770}. For model-based methods, we tuned parameters for optimal performance. For Joint-PnP~\cite{doi:10.1137/24M1677770} and LPR-Net~\cite{10.1007/978-3-031-92366-1_30}, we use the official pretrained weights released by the respective authors. All three learning-based methods (including ours) are trained on synthetic data generated by the same procedure~\cite{doi:10.1137/24M1677770}. Our source code, pretrained weights, and training data are publicly available.\footnote{\url{https://github.com/yangli161029/NGVD}}

\begin{figure}[tbp]
\centering
\begingroup
\presetFive
\setlength{\tabcolsep}{1pt}
\renewcommand{\arraystretch}{1.0}

\begin{tabular}{@{}p{\colw} p{\colw} p{\colw} p{\colw} p{\colw}@{}}
  \multirow{2}{*}{\raisebox{-0.5\height}{\cellimg{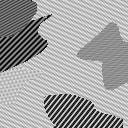}{}}} &
  \cellimgTR{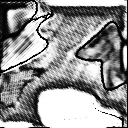}{\(W_1\)} &
  \cellimgTR{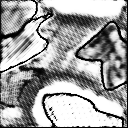}{\(W_1\)} &
  \cellimgTR{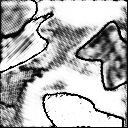}{\(W_1\)} &
  \cellimgTR{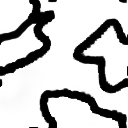}{\(W_1\) (prob.)} \\
  &
  \cellimgTR{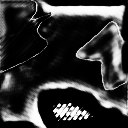}{\(W_2\)} &
  \cellimgTR{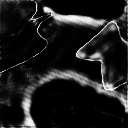}{\(W_2\)} &
  \cellimgTR{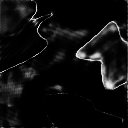}{\(W_2\)} &
  \cellimgTR{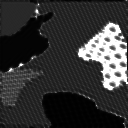}{\(W_2\)} \\
  \noalign{\vskip\dimexpr6\tabcolsep\relax} 

  \multirow{2}{*}{\raisebox{-0.5\height}{\cellimg{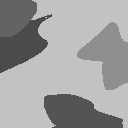}{}}} &
  \cellimgTR{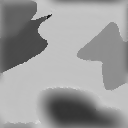}{PSNR: 23.45} &
  \cellimgTR{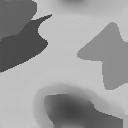}{PSNR: 25.59} &
  \cellimgTR{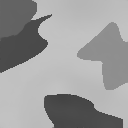}{PSNR: 31.71} &
  \cellimgTR{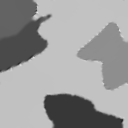}{PSNR: 26.69} \\
  &
  \cellimgTR{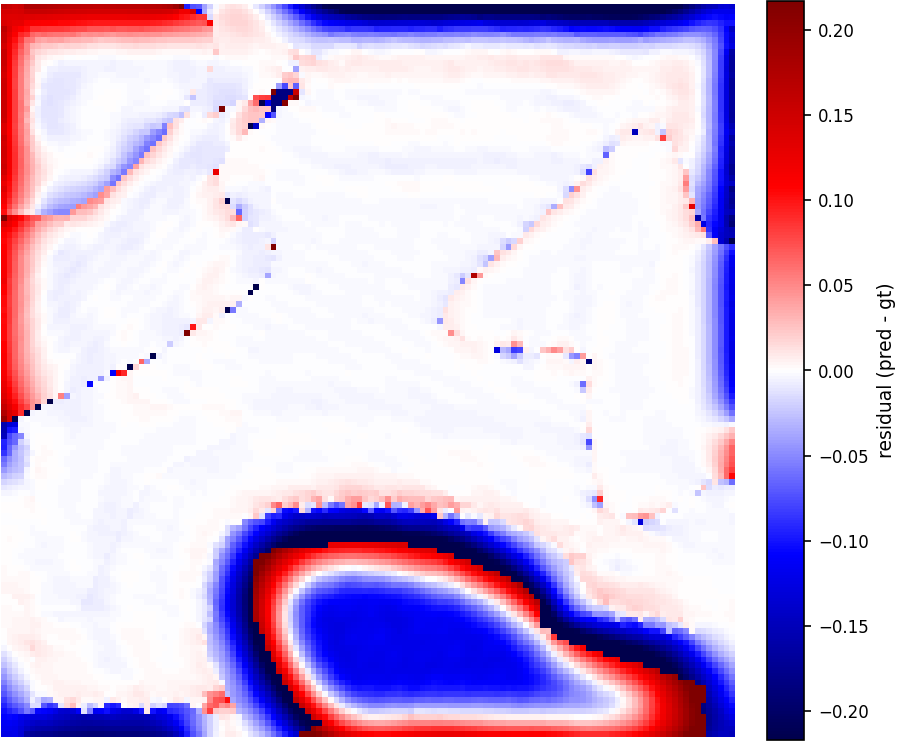}{RMSE: 0.07} &
  \cellimgTR{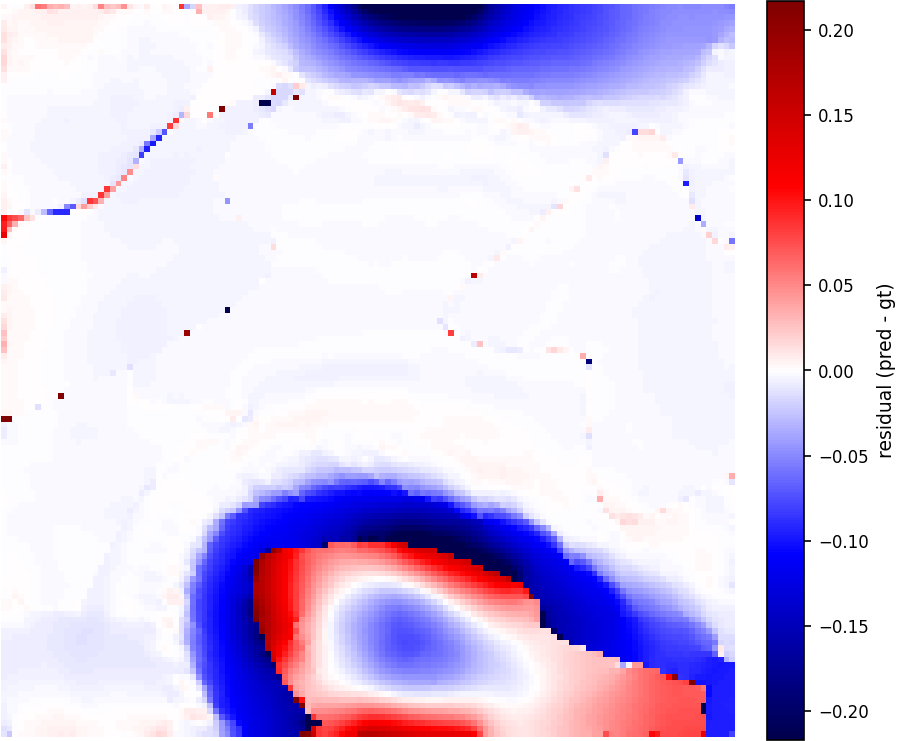}{RMSE: 0.05} &
  \cellimgTR{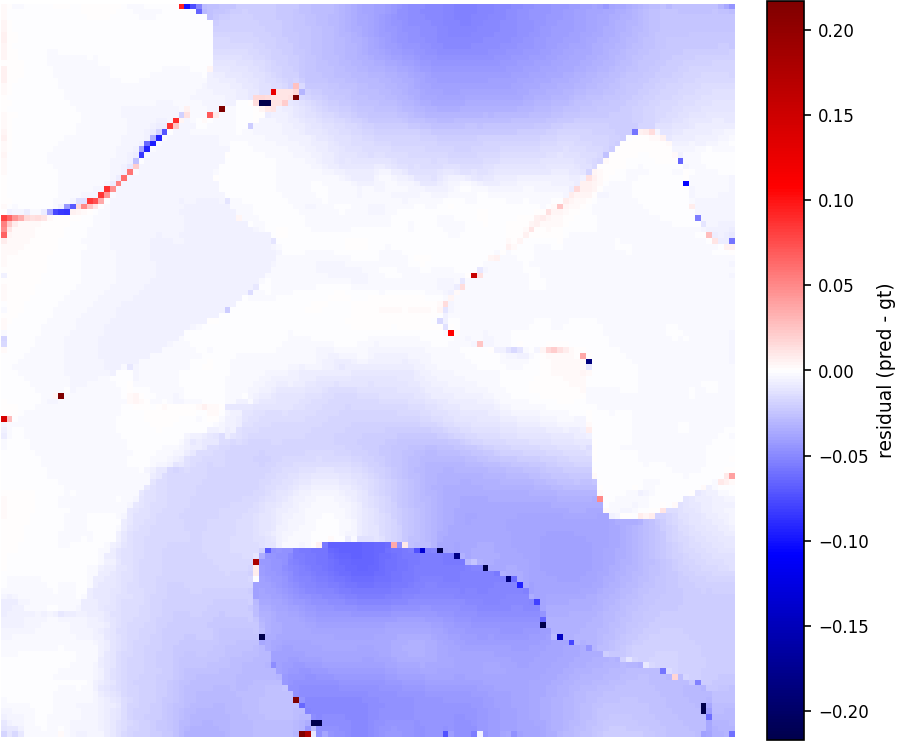}{RMSE: 0.03} &
  \cellimgTR{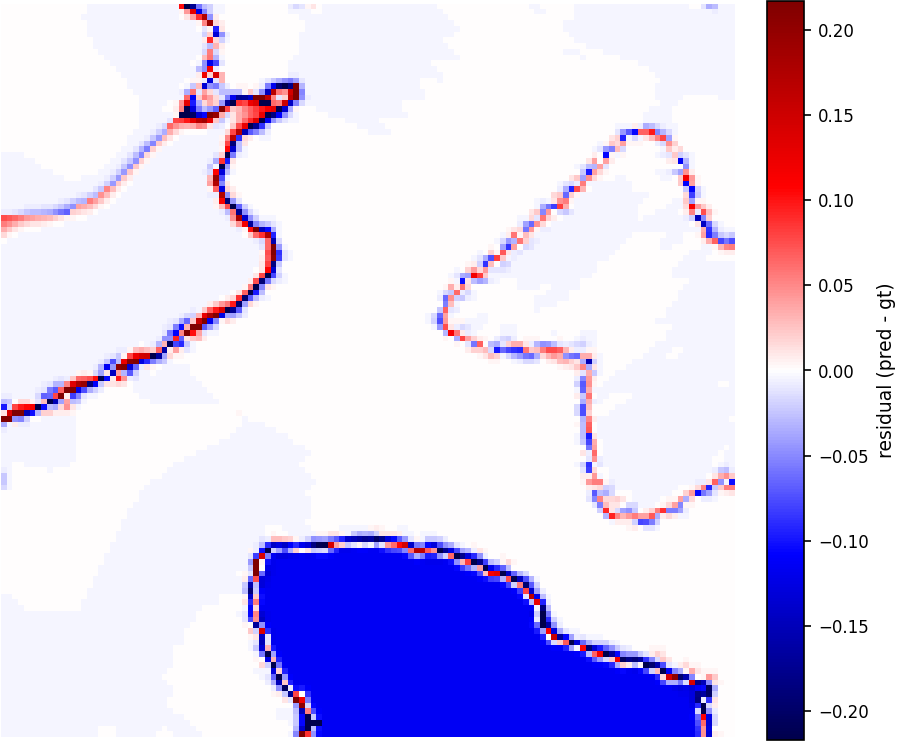}{RMSE: 0.05} \\
  \noalign{\vskip\dimexpr6\tabcolsep\relax} 

  \multirow{2}{*}{\raisebox{-0.5\height}{\cellimg{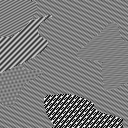}{}}} &
  \cellimgTR{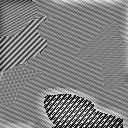}{PSNR: 23.48} &
  \cellimgTR{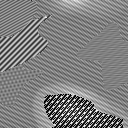}{PSNR: 25.60} &
  \cellimgTR{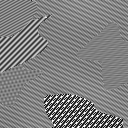}{PSNR: 31.71} &
  \cellimgTR{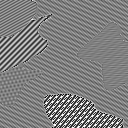}{PSNR: 26.95} \\
  &
  \cellimgTR{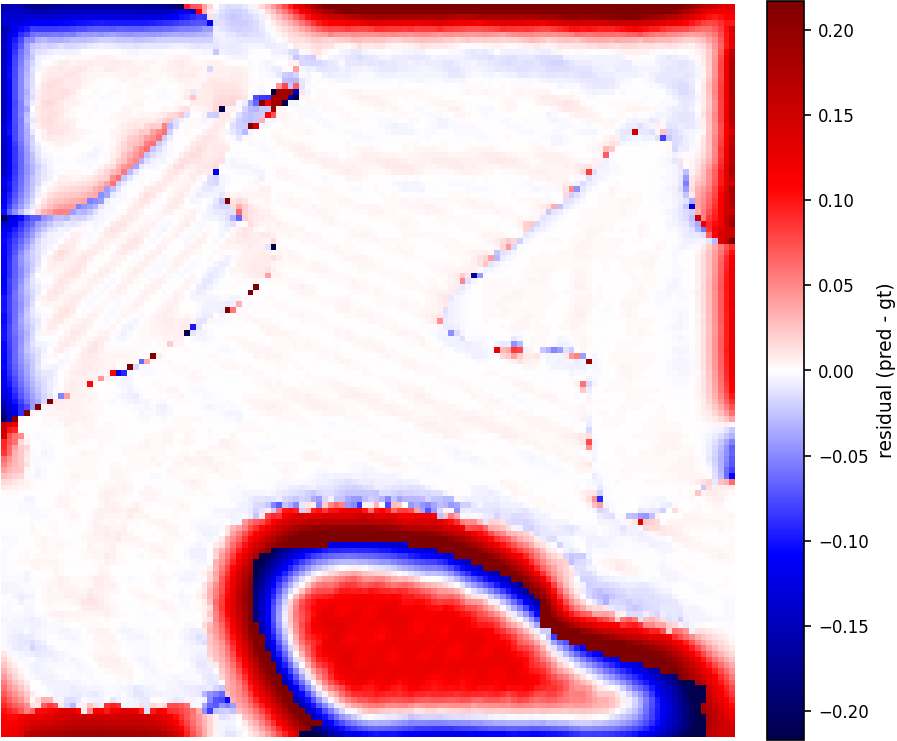}{RMSE: 0.07} &
  \cellimgTR{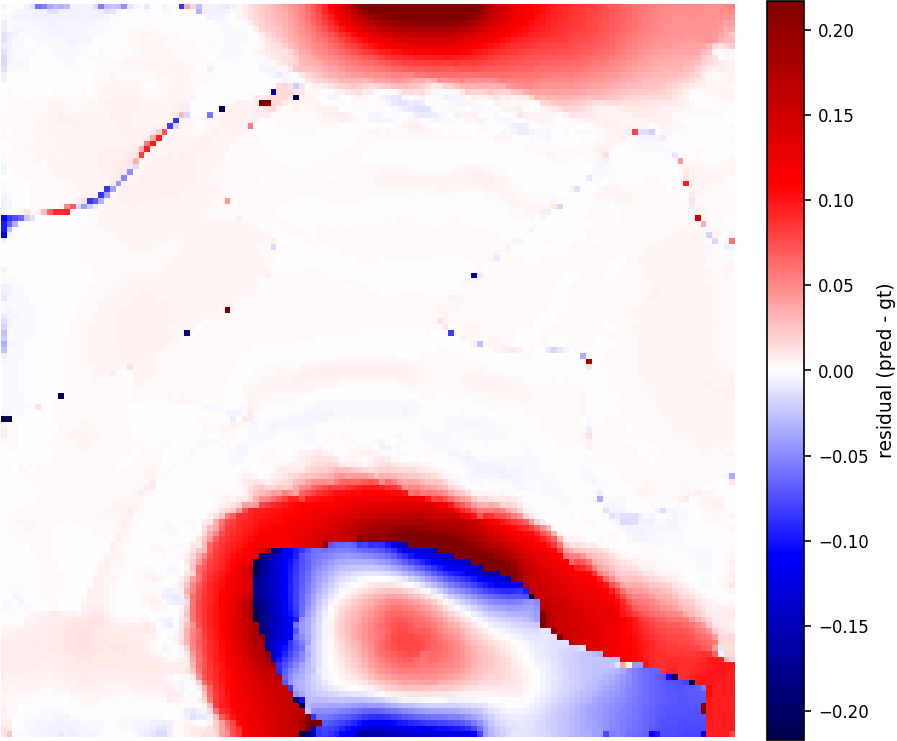}{RMSE: 0.05} &
  \cellimgTR{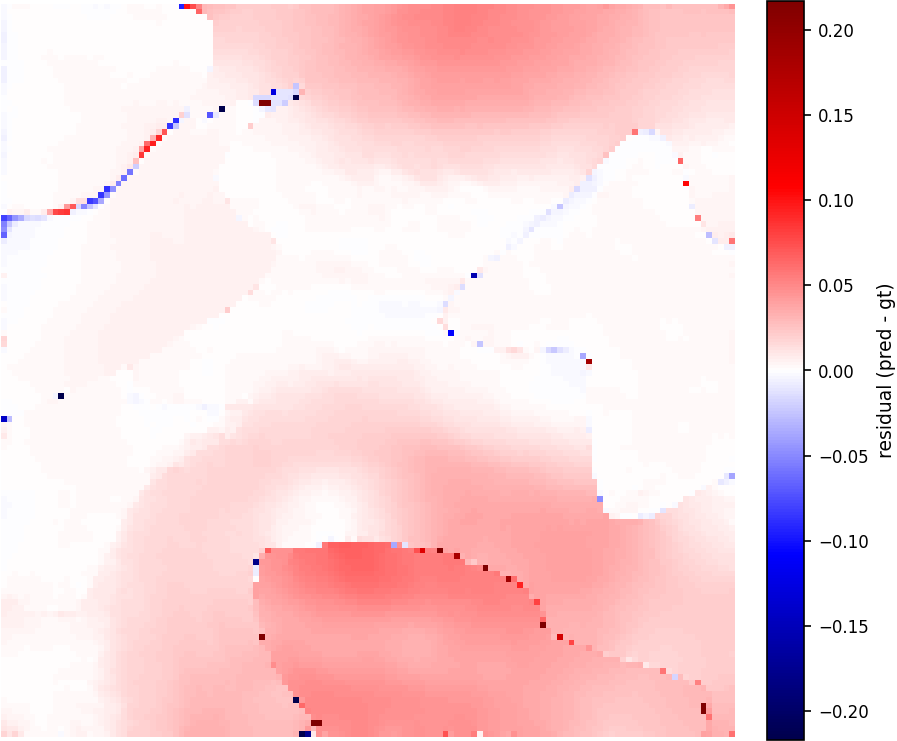}{RMSE: 0.03} &
  \cellimgTR{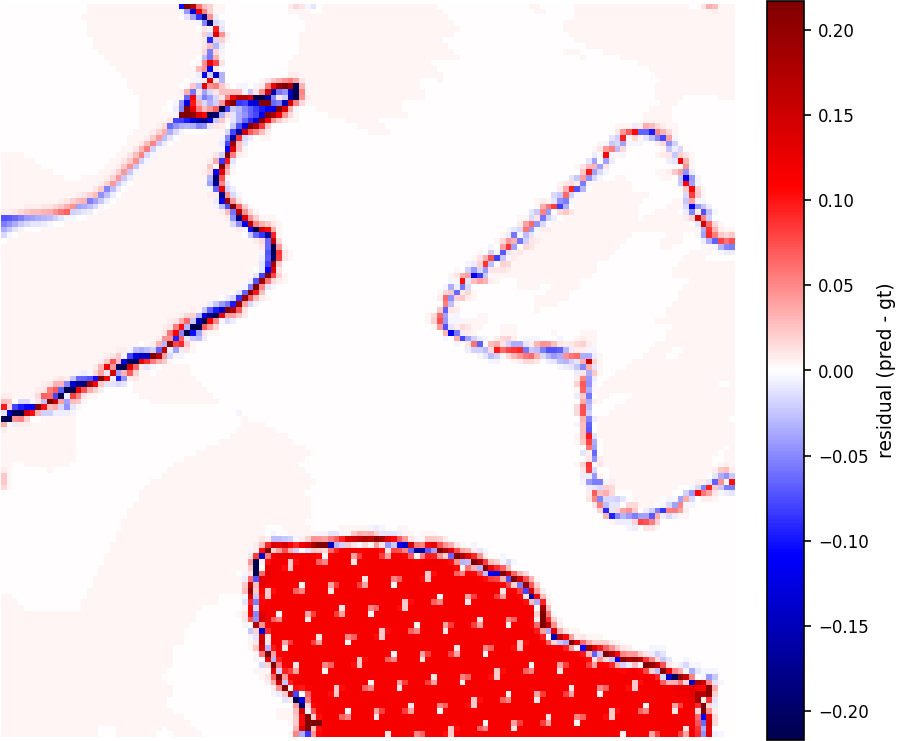}{RMSE: 0.05} \\
\end{tabular}

\vspace{6pt}

\begin{tabular}{@{}p{\colw} p{\colw} p{\colw} p{\colw} p{\colw}@{}}
  \centering \textsf{Reference} &
  \centering \textsf{\(k = 1\)} &
  \centering \textsf{\(k = 4\)} &
  \centering \textsf{\(k = 8\)} &
  \centering \textsf{Prob.} \\
\end{tabular}

\caption{Example~1 -- Iterative image decomposition for a synthetic image and probabilistic baseline. First column: observed image, ground-truth cartoon, and ground-truth texture (each spanning two rows). Columns two to four: results at iterations \(k = 1,4,8\) of the proposed method, showing from top to bottom \(W_1\), \(W_2\), reconstructed cartoon (with PSNR), cartoon residual (with RMSE), reconstructed texture (with PSNR), and texture residual (with RMSE). The last column reports the same quantities for the probabilistic baseline. }
\label{fig:iter-decomp}
\endgroup
\end{figure}

For our model, the regularization parameters are initialized as $\lambda_1 = 1$ and $\lambda_2 = 0.2$, with adaptive weight updates conducted over $K=8$ iterations unless otherwise specified.
The MLP $\Lambda_{\Theta_1}$ is initialized: the last
linear layer has zero weights and biases $\mathrm{softplus}^{-1}(\lambda_i^{(0)})$.
The U-Net $\mathcal{W}_{\Theta_2}$ uses Kaiming uniform initialization,
with the final sigmoid layer biased toward $0.5$.
The inner CG solver runs for at most $T=80$ steps per outer iteration
(warm-started from $\hat{x}_{k-1}$, tolerance $10^{-6}$).
Both subnetworks are optimized with Adam and a \texttt{ReduceLROnPlateau}
scheduler; training 600 epochs (batch size 24) on a single NVIDIA RTX~4090
takes approximately 2 hours.

Quantitative accuracy evaluations are based on the peak signal-to-noise ratio (PSNR), root mean squared error (RMSE), and structural similarity index measure (SSIM). Given the ground-truth cartoon component $c^\ast$ and its estimate $\hat{c}$ (similarly for the texture component, $t^\ast$ and $\hat{t}$), the RMSE and PSNR metrics are defined as 
$$\mathrm{RMSE}(c^\ast, \hat{c}) = \sqrt{\frac{1}{N} \sum_{i=1}^{N} (\hat{c}_i - c_i^\ast)^2}, \quad\;\:
\mathrm{PSNR}(c^\ast, \hat{c}) = 20 \log_{10} \left( \frac{1}{\mathrm{RMSE}(c^\ast, \hat{c})} \right),$$
where $ N $ denotes the number of pixels. Higher PSNR and lower RMSE values indicate superior reconstruction accuracy. The SSIM, which assesses perceptual quality, is computed as
$$ \mathrm{SSIM}(c^\ast, \hat{c}) = \frac{(2\mu_{c^\ast} \mu_{\hat{c}} + \epsilon_1)(2\sigma_{c^\ast \hat{c}} + \epsilon_2)}{(\mu_{c^\ast}^2 + \mu_{\hat{c}}^2 + \epsilon_1)(\sigma_{c^\ast}^2 + \sigma_{\hat{c}}^2 + \epsilon_2)}, $$
where $ \mu_c $ and $ \sigma_c $ represent the mean and (co)variance, respectively, and $ \epsilon_1, \epsilon_2 $ are small constants for stabilization. SSIM values closer to 1 reflect better preservation of structural information.

\subsection{Example 1: Iterative Scheme and Adaptive Weights}

This section validates the iterative scheme of the proposed method, with a focus on the adaptive weight updates that refine the decomposition of an observed image into cartoon and texture components. At each outer iteration, the algorithm first updates the weights \( W_1 \) and \( W_2 \), and then solves the associated quadratic problem for the cartoon--texture pair. The weights act as adaptive masks that separate structural and textural features. We analyze the progression across iterations, presenting visual and quantitative results at iterations \(k = 1, 4, 8\) for a representative synthetic image, and compare them with a probabilistic baseline that uses the same variational model but updates \(W_1\) and \(W_2\) by a probabilistic rule (see Section~\ref{sec:PGVD}).

Fig.~\ref{fig:iter-decomp} presents the decomposition results. The first column presents the observed image \(f\), ground-truth cartoon \(c^\star\), and ground-truth texture \(t^\star\) (each spanning two rows). The next three columns report the results of the proposed method at iterations \(k=1,4,8\), including learned weights \(W_1\) and \(W_2\), reconstructed cartoon and texture (with PSNR), and residuals (with RMSE). The last column reports the same quantities for the probabilistic baseline. All residual maps are visualized using a zero-centered diverging colorbar: saturated red/blue indicate larger positive/negative errors, while white corresponds to small residuals. 

In the first outer iteration, the decomposition remains coarse—region boundaries are not precisely located and high-frequency content leaks into the cartoon. This is reflected by the relatively low PSNR and more intense residual maps. As the iterations proceed, both the cartoon and texture components become cleaner and sharper, with residuals fading toward white. Quantitatively, the PSNR improves from 23.45/23.48 (cartoon/texture) at \(k=1\) to 31.71 for both components at \(k=8\), while RMSE drops from 0.07 to 0.03.

Compared with the learned-weight method, the probabilistic baseline yields globally consistent but less refined results. The cartoon boundaries are slightly blurred and fine-scale oscillatory textures are partially lost. These limitations are confirmed by the residual maps, where strong positive/negative deviations persist along edges and inside the high-frequency patch. In contrast, the learned-weight residuals in the same regions are nearly white. The PSNR of the probabilistic method (26.69/26.95 for cartoon/texture) improves upon the initial network iteration but remains clearly inferior to the final result.

To evaluate the impact of directional specificity, we conduct an ablation study comparing isotropic weights (\( W_x = W_y \) for both \( W_1 \) and \( W_2 \)) against an anisotropic variant (distinct weights in the \( x \) and \( y \) directions, described in the general model~\eqref{General_obj}). The anisotropic model is trained similarly, generating separate directional weight maps.

\begin{figure}[tbp]
\centering
\begingroup
  \presetThree

  \localset{0.131}{0.98}{0.88pt} 
  \renewcommand{\arraystretch}{0.9}

  \begin{tabular}{@{} p{\colw} p{\colw} p{\colw} p{\colw} p{\colw} p{\colw} p{\colw} @{}}

  \multirow{2}{*}{\raisebox{-0.5\height}{%
    \cellimgTR{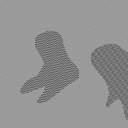}{Observed}}} &
  \cellimgTR{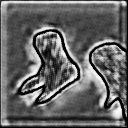}{Isotropic \(W_1\)} &
  \cellimgTR{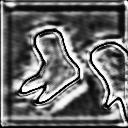}{\(W_{1,x}\)} &
  \cellimgTR{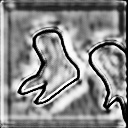}{\(W_{1,y}\)} &
  \cellimgTR{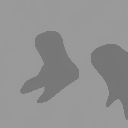}{SSIM 0.9987} &
  \cellimgTR{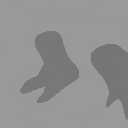}{SSIM 0.9982} &
  \cellimgTR{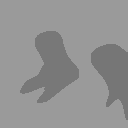}{} \\

  &
  \cellimgTR{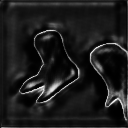}{Isotropic \(W_2\)} &
  \cellimgTR{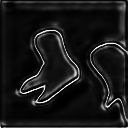}{\(W_{2,x}\)} &
  \cellimgTR{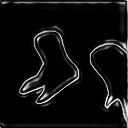}{\(W_{2,y}\)} &
  \cellimgTR{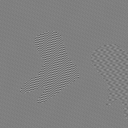}{SSIM 0.9669} &
  \cellimgTR{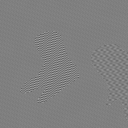}{SSIM 0.9657} &
  \cellimgTR{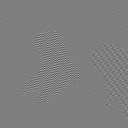}{} \\

  \noalign{\vskip\dimexpr2\tabcolsep\relax}

  \multicolumn{4}{@{}c@{}}{} &
  \centering {\fontsize{8.0}{8.2}\selectfont\textsf{\mbox{Isotropic}}} &
  \centering {\fontsize{8.0}{8.2}\selectfont\textsf{\mbox{Anisotropic}}} &
  \centering {\fontsize{8.0}{8.2}\selectfont\textsf{\mbox{Reference}}} \\
  
  \end{tabular}

  \caption{Example 1- Comparison of isotropic vs.\ anisotropic weights. Left: observed and weight maps (isotropic \(W_1/W_2\); anisotropic split into \(x/y\)). Right: reconstructed cartoon/texture with overlaid SSIM.}
  \label{fig:aniso-comp}
\endgroup
\end{figure}

\begin{figure}[tbp]
\centering
\begingroup
  \presetSeven

  \setlength{\tabcolsep}{0.9pt}

  \newcommand{\rowsepeqcol}{\noalign{\vskip\dimexpr2\tabcolsep\relax}}

  \begin{tabular}{@{} p{\colw} p{\colw} p{\colw} p{\colw} p{\colw} p{\colw} p{\colw} p{\colw} @{}}
  \multirow{2}{*}{\raisebox{-0.5\height}{\cellimg{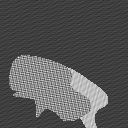}{}}} &
  \cellimgTR{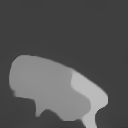}{PSNR: 32.87} &
  \cellimgTR{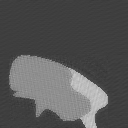}{PSNR: 34.08} &
  \cellimgTR{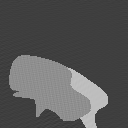}{PSNR: 39.79} &
  \cellimgTR{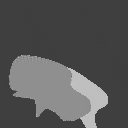}{PSNR: 41.10} &
  \cellimgTR{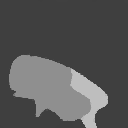}{PSNR: 37.65} &
  \cellimgTR{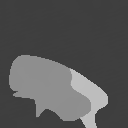}{PSNR: 44.16} &
  \cellimgTR{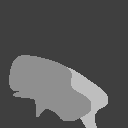}{} \\
  &
  \cellimgTR{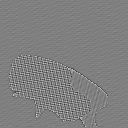}{PSNR: 32.68} &
  \cellimgTR{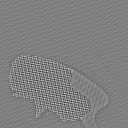}{PSNR: 34.02} &
  \cellimgTR{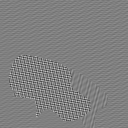}{PSNR: 39.64} &
  \cellimgTR{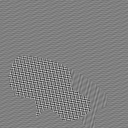}{PSNR: 40.89} &
  \cellimgTR{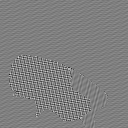}{PSNR: 37.65} &
  \cellimgTR{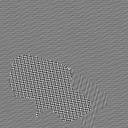}{PSNR: 44.16} &
  \cellimgTR{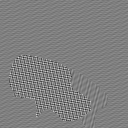}{} \\
  \rowsepeqcol
  \end{tabular}

  \begin{tabular}{@{} p{\colw} p{\colw} p{\colw} p{\colw} p{\colw} p{\colw} p{\colw} p{\colw} @{}}
  \multirow{2}{*}{\raisebox{-0.5\height}{\cellimg{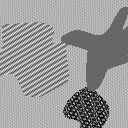}{}}} &
  \cellimgTR{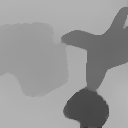}{PSNR: 9.26} &
  \cellimgTR{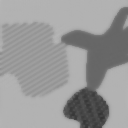}{PSNR: 29.14} &
  \cellimgTR{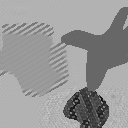}{PSNR: 26.81} &
  \cellimgTR{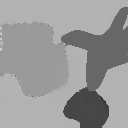}{PSNR: 32.14} &
  \cellimgTR{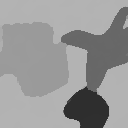}{PSNR: 34.31} &
  \cellimgTR{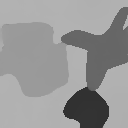}{PSNR: 39.22} &
  \cellimgTR{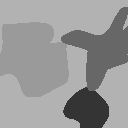}{} \\
  &
  \cellimgTR{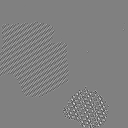}{PSNR: 21.84} &
  \cellimgTR{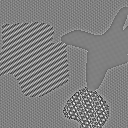}{PSNR: 23.73} &
  \cellimgTR{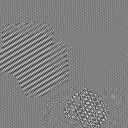}{PSNR: 23.38} &
  \cellimgTR{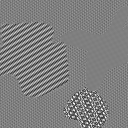}{PSNR: 24.15} &
  \cellimgTR{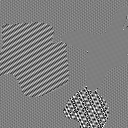}{PSNR: 34.31} &
  \cellimgTR{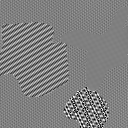}{PSNR: 39.22} &
  \cellimgTR{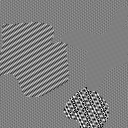}{} \\
  \rowsepeqcol
  \end{tabular}

  \begin{tabular}{@{} p{\colw} p{\colw} p{\colw} p{\colw} p{\colw} p{\colw} p{\colw} p{\colw} @{}}
  \multirow{2}{*}{\raisebox{-0.5\height}{\cellimg{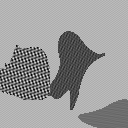}{}}} &
  \cellimgTR{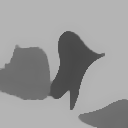}{PSNR: 8.75} &
  \cellimgTR{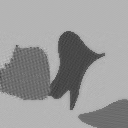}{PSNR: 33.00} &
  \cellimgTR{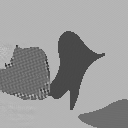}{PSNR: 31.89} &
  \cellimgTR{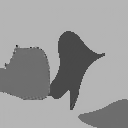}{PSNR: 38.40} &
  \cellimgTR{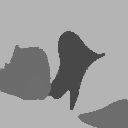}{PSNR: 37.55} &
  \cellimgTR{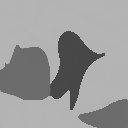}{PSNR: 41.33} &
  \cellimgTR{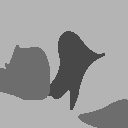}{} \\
  &
  \cellimgTR{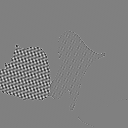}{PSNR: 20.44} &
  \cellimgTR{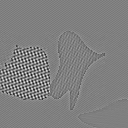}{PSNR: 31.04} &
  \cellimgTR{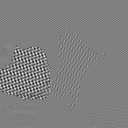}{PSNR: 30.21} &
  \cellimgTR{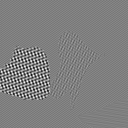}{PSNR: 33.27} &
  \cellimgTR{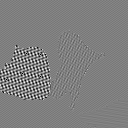}{PSNR: 37.55} &
  \cellimgTR{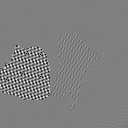}{PSNR: 41.32} &
  \cellimgTR{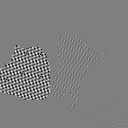}{} \\
  \end{tabular}

  \vspace{6pt}

\begin{tabular}{@{} p{\colw} p{\colw} p{\colw} p{\colw} p{\colw} p{\colw} p{\colw} p{\colw} @{}}
  & \centering {\fontsize{7.0}{8.2}\selectfont\textsf{\mbox{TV-Gnorm~\cite{wen2019primal}}}} %
  & \centering {\fontsize{7.0}{8.2}\selectfont\textsf{\mbox{LR-WLS~\cite{li2025cartoon}}}} %
  & \centering {\fontsize{7.0}{8.2}\selectfont\textsf{\mbox{Joint-PnP~\cite{doi:10.1137/24M1677770}}}} %
  & \centering {\fontsize{7.0}{8.2}\selectfont\textsf{\mbox{LPR-Net~\cite{10.1007/978-3-031-92366-1_30}}}} %
  & \centering {\fontsize{7.0}{8.2}\selectfont\textsf{\mbox{Ours(PGVD)}}} %
  & \centering {\fontsize{7.0}{8.2}\selectfont\textsf{\mbox{Ours(NGVD)}}} %
  & \centering {\fontsize{7.0}{8.2}\selectfont\textsf{\mbox{GT}}} \\
\end{tabular}

  \caption{Example~2 - Image decomposition results. Left column: observed input (spanning two rows). Top row per sample: reconstructed cartoon; bottom row: reconstructed texture. PSNR values overlaid at higher-left of each method result (ground truth column intentionally has no PSNR).}
  \label{fig:decomp-grid}
\endgroup
\end{figure}

Fig.~\ref{fig:aniso-comp} compares the weight maps and final reconstructions, with metrics overlaid. Quantitative results on the synthetic dataset show negligible differences: the isotropic model achieves SSIM values of 0.9987 (cartoon) and 0.9669 (texture), while the anisotropic model yields 0.9982 and 0.9657, respectively (differences \(< 0.002\)). Visual inspections reveal nearly identical reconstructions, with the anisotropic weights displaying subtle directional variations but no significant improvements. Given the absence of strong orientational biases in the datasets, we adopt isotropic weights for simplicity in subsequent experiments.

\subsection{Example~2: Comparison with State-of-the-Art Methods}

In this subsection, we compare our method against the baselines on both synthetic and real-world datasets. For synthetic data, quantitative metrics are computed using the ground truth, whereas real-world evaluations rely on visual inspections. For the model-based baselines on all
real-world images, TV-Gnorm uses $(\alpha,\lambda)=(20,0.07)$ following
the notation of~\cite{wen2019primal}, and LR-WLS uses
$(\tau,\beta)=(1.07,10^{-3})$ following
~\cite{li2025cartoon}.

Fig.~\ref{fig:decomp-grid} presents decomposition results for three representative synthetic samples. Each block displays the observed input (left, spanning two rows), followed by the reconstructed cartoon (top) and texture (bottom) for TV-Gnorm \cite{wen2019primal}, LR-WLS \cite{li2025cartoon}, Joint-PnP \cite{doi:10.1137/24M1677770}, LPR-Net \cite{10.1007/978-3-031-92366-1_30}, our method, and the ground truth. PSNR values are overlaid on the method outputs.

Our approach consistently outperforms the competitors across all synthetic samples, achie-ving the highest PSNR values and demonstrating superior separation of structural and textural components, with reconstructions closely approximating the ground truth. In contrast, TV-Gnorm yields the lowest PSNR, primarily due to over-smoothing that blurs edges in cartoons and textures. LR-WLS and Joint-PnP improve modestly  but shows boundary fuzziness in cartoons and incomplete texture isolation, allowing textural details to bleed into structural parts.  LPR-Net performs better by achieving basic separation, yet it struggles with fine-grained details at edges, resulting in artifacts such as residual patterns in cartoons or incomplete texture capture, particularly in complex patterns. Our method's adaptive weighting ensures cartoons with sharp contours and well-isolated textures, free of boundary artifacts, across varied synthetic scenarios.

\begin{figure}[tbp]
\centering
\begingroup
  \presetSix
  \begin{tabular}{@{} p{\colw} p{\colw} p{\colw} p{\colw} p{\colw} p{\colw} @{}}
  \multirow{2}{*}{\raisebox{-0.5\height}{\cellimg{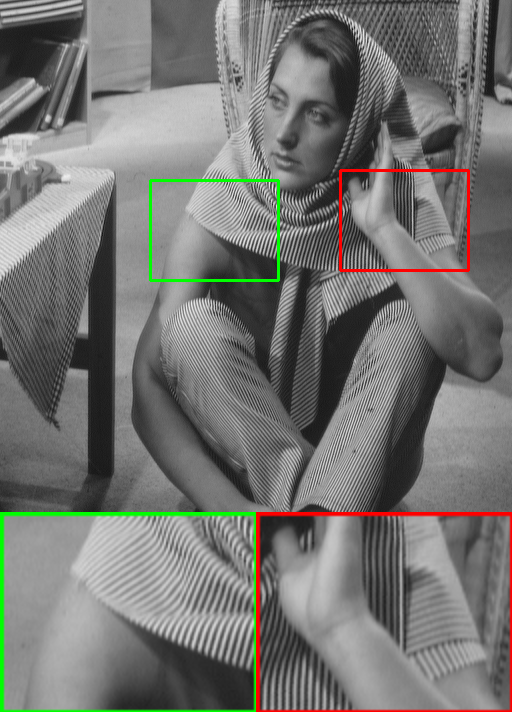}{}}} &
  \cellimg{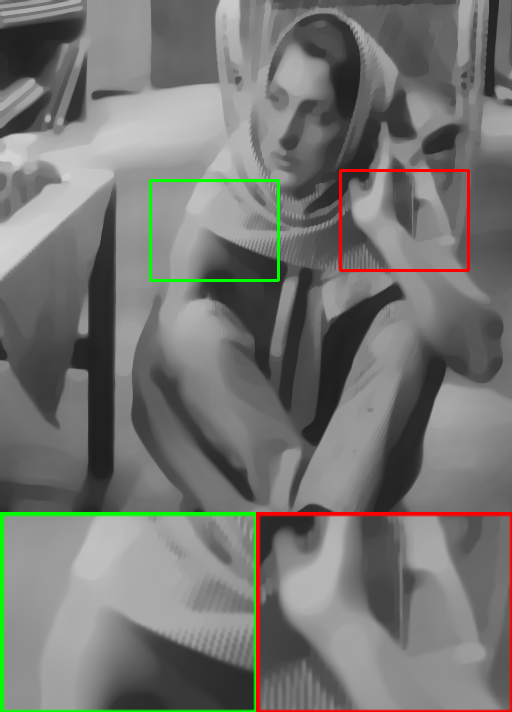}{} &
  \cellimg{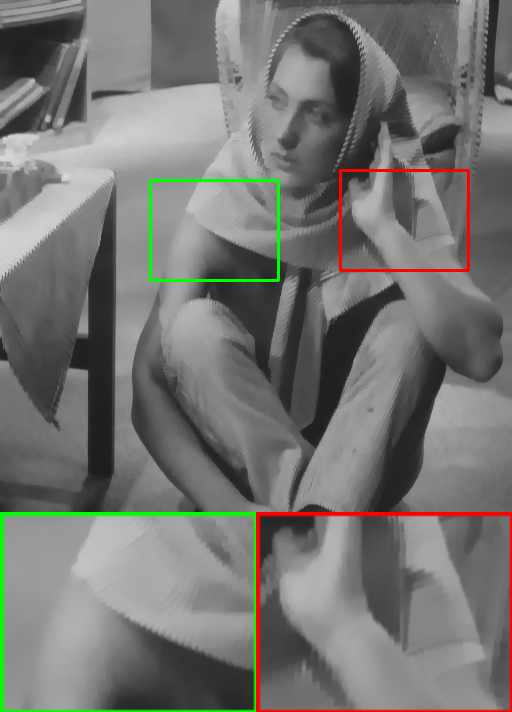}{} &
  \cellimg{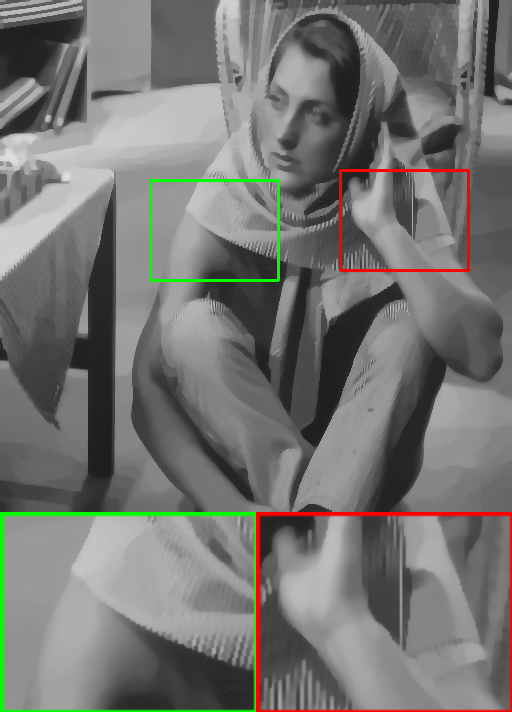}{} &
  \cellimg{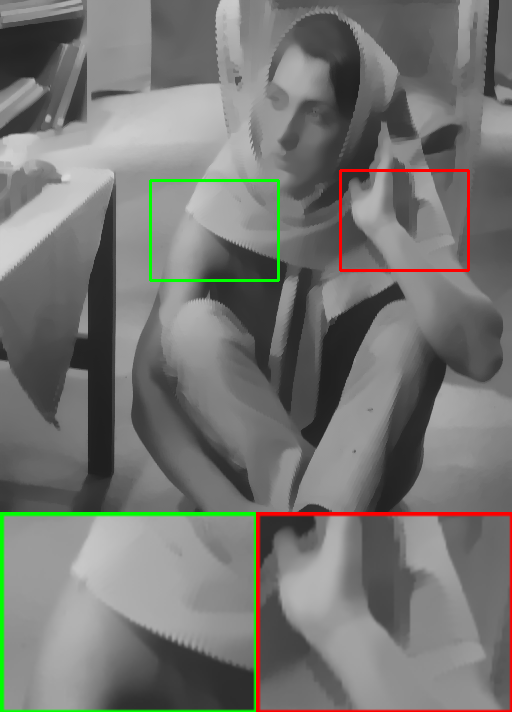}{} &
  \cellimg{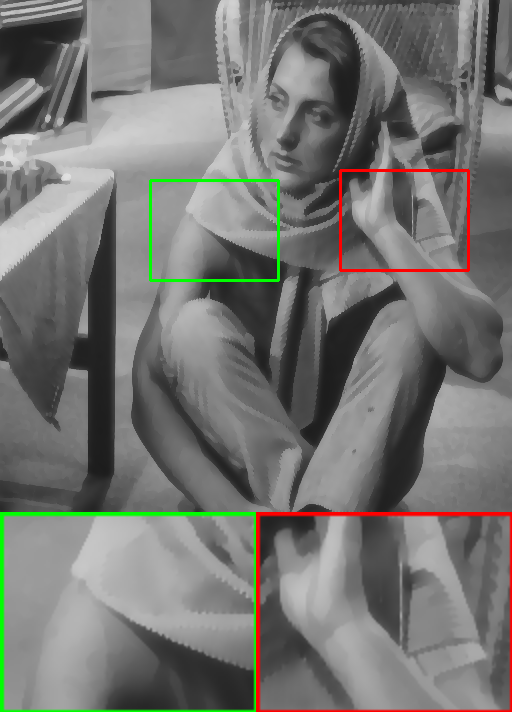}{} \\
  &
  \cellimg{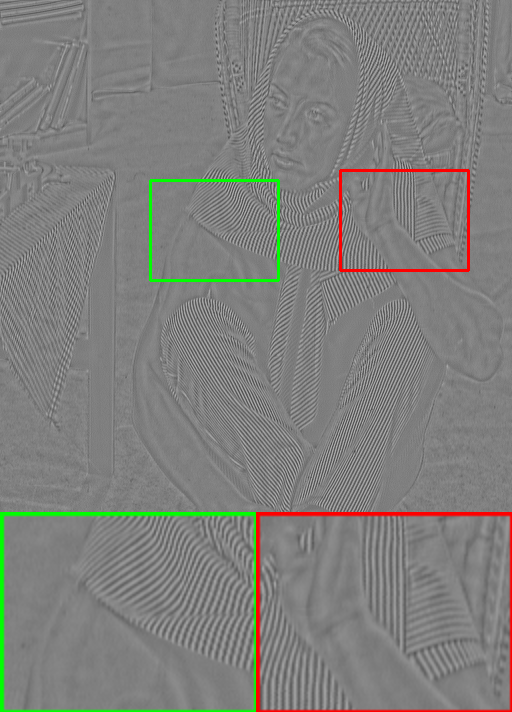}{} &
  \cellimg{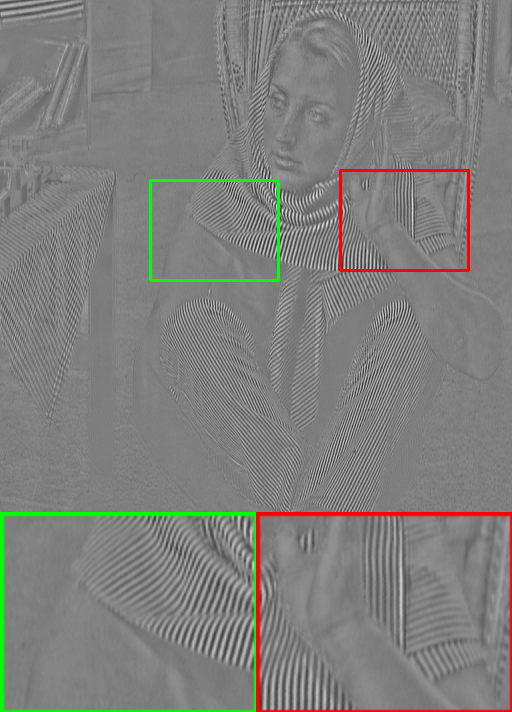}{} &
  \cellimg{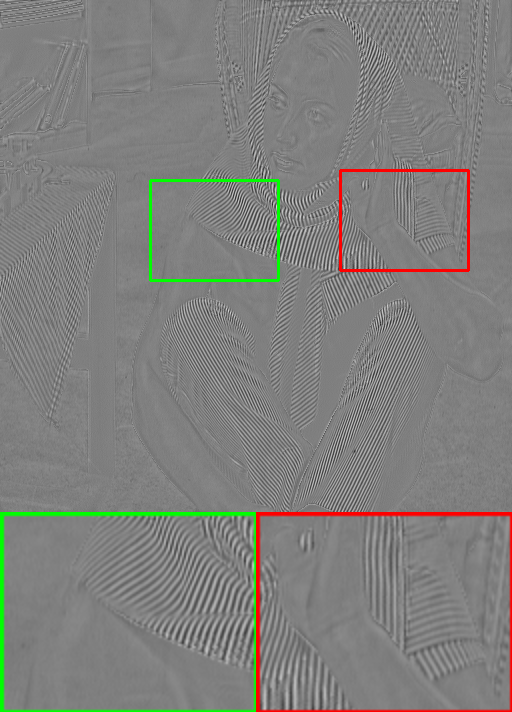}{} &
  \cellimg{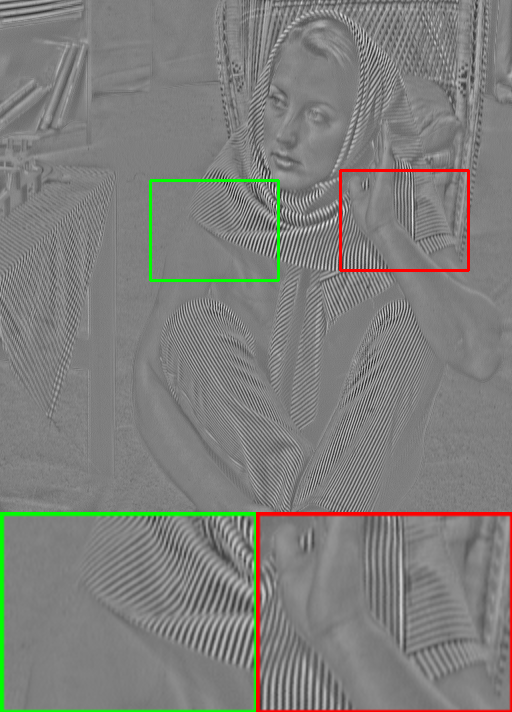}{} &
  \cellimg{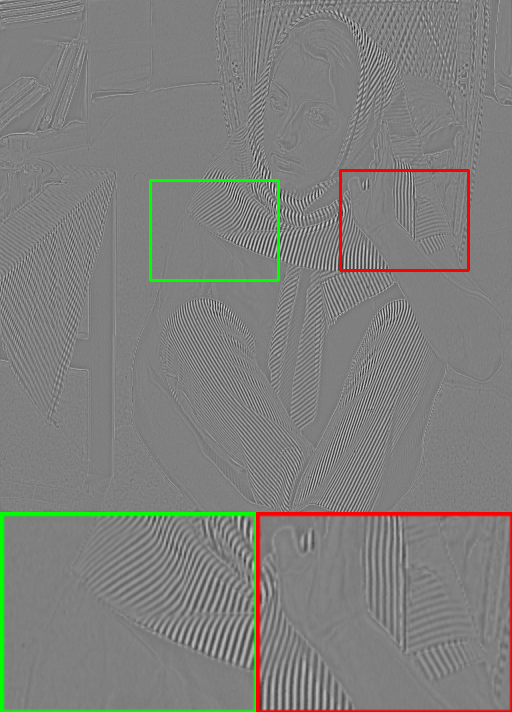}{} \\
  \end{tabular}

  \vspace{4pt} 

  \begin{tabular}{@{} p{\colw} p{\colw} p{\colw} p{\colw} p{\colw} p{\colw} @{}}
  \multirow{2}{*}{\raisebox{-0.5\height}{\cellimg{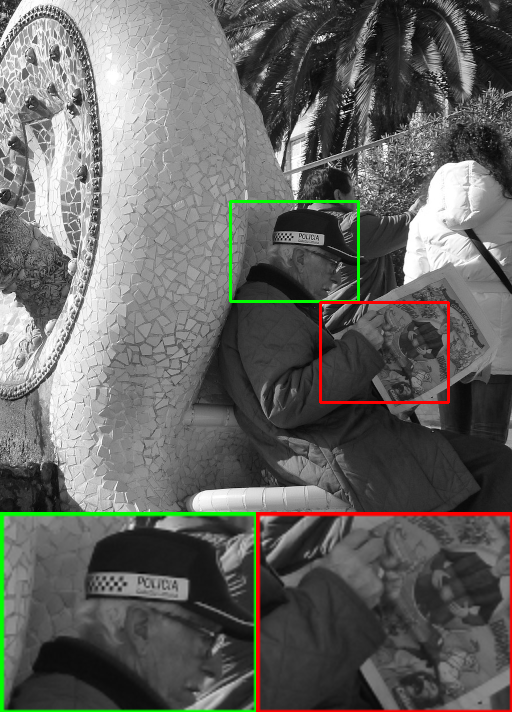}{}}} &
  \cellimg{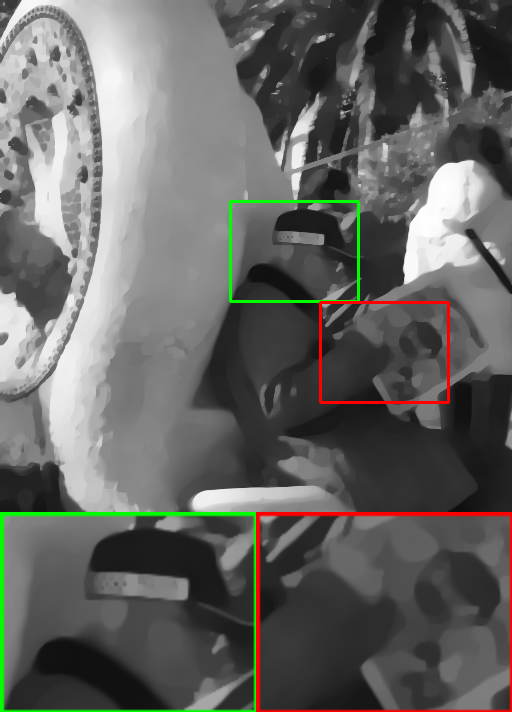}{} &
  \cellimg{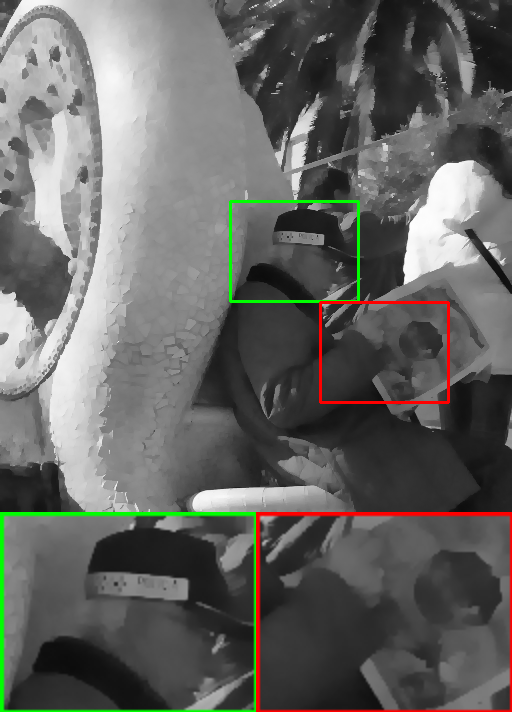}{} &
  \cellimg{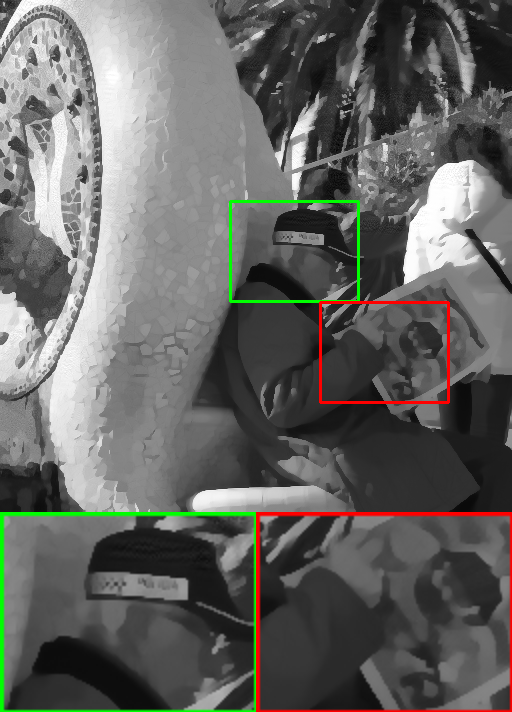}{} &
  \cellimg{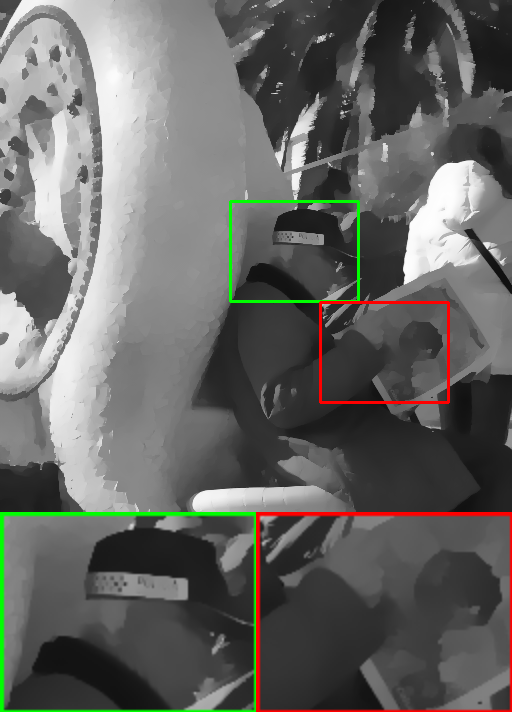}{} &
  \cellimg{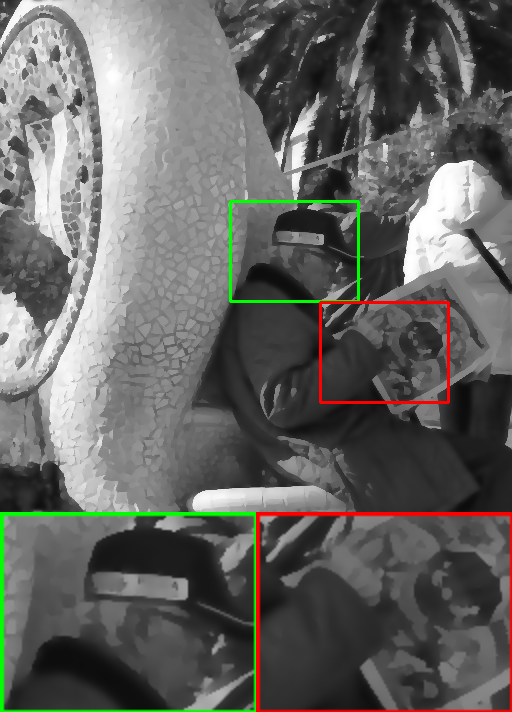}{} \\
  &
  \cellimg{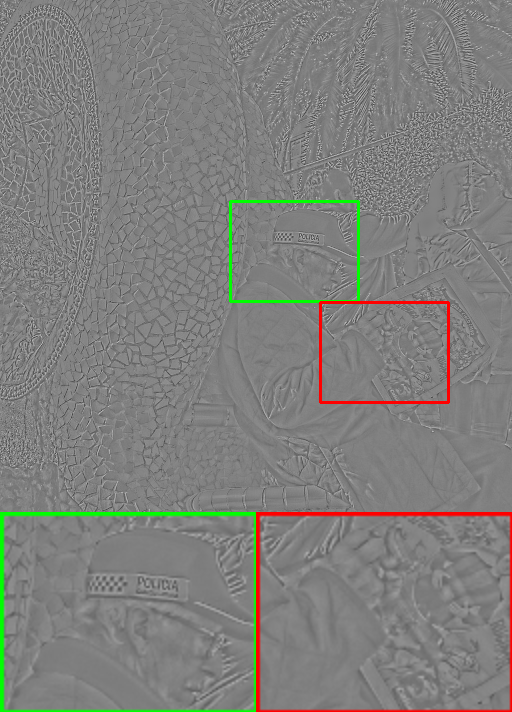}{} &
  \cellimg{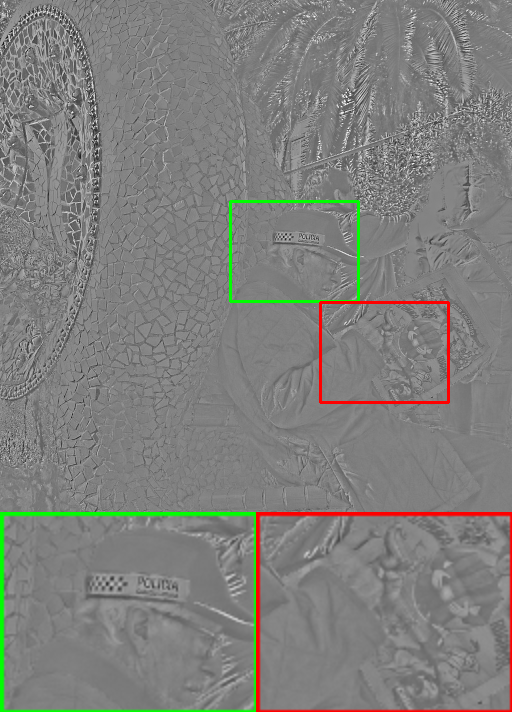}{} &
  \cellimg{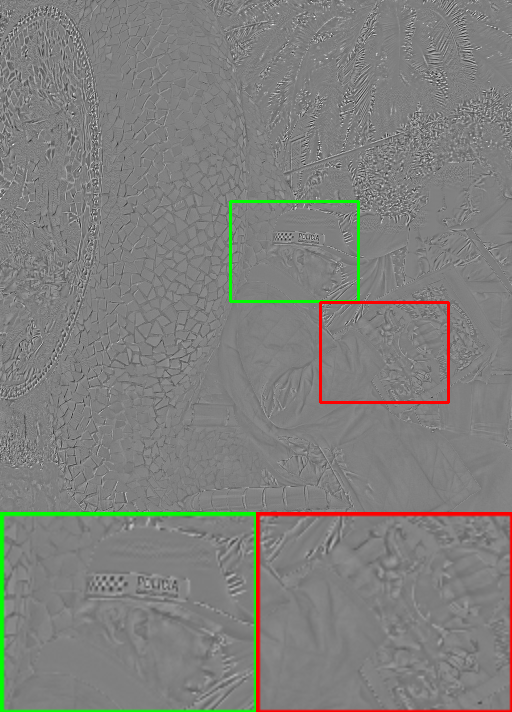}{} &
  \cellimg{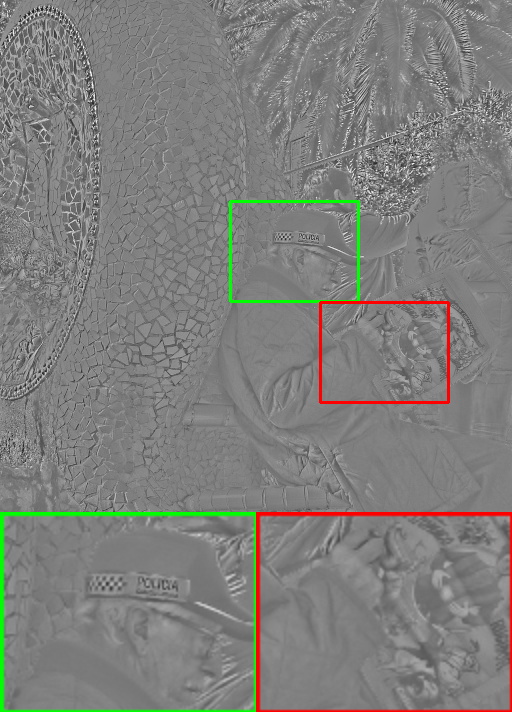}{} &
  \cellimg{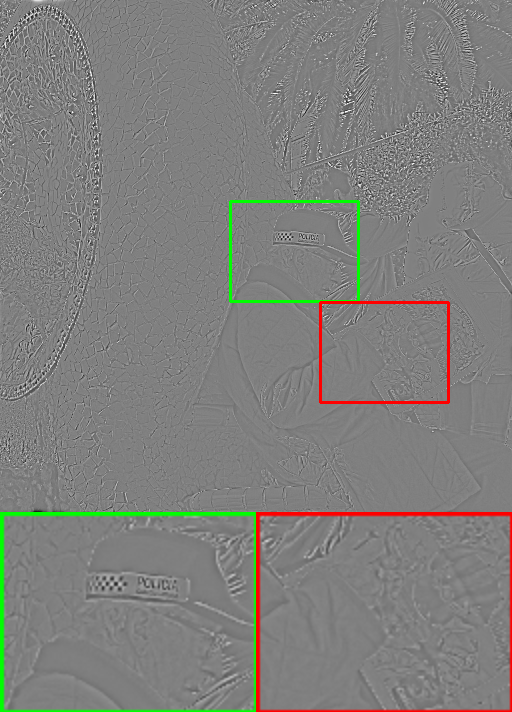}{} \\
  \end{tabular}

  \vspace{6pt}

  \begin{tabular}{@{} p{\colw} p{\colw} p{\colw} p{\colw} p{\colw} p{\colw} @{}}
  & \centering {\fontsize{7.0}{8.2}\selectfont\textsf{\mbox{TV-Gnorm \cite{wen2019primal}}}} &
  \centering {\fontsize{7.0}{8.2}\selectfont\textsf{\mbox{LR-WLS \cite{li2025cartoon}}}} &
  \centering {\fontsize{7.0}{8.2}\selectfont\textsf{\mbox{Joint-PnP \cite{doi:10.1137/24M1677770}}}} &
  \centering {\fontsize{7.0}{8.2}\selectfont\textsf{\mbox{LPR-Net \cite{10.1007/978-3-031-92366-1_30}}}} &
  \centering {\fontsize{7.0}{8.2}\selectfont\textsf{\mbox{Ours(NGVD)}}} \\
  \end{tabular}

  \caption{Example~2 - Real-world image decomposition results. Left column: observed input (spanning two rows, with zoomed regions integrated below). Top row per sample: cartoon component; bottom row: texture component. }
  \label{fig:natural_results}
\endgroup
\end{figure}

To complement the per-image comparisons, Table~\ref{tab:full_comparison} reports average PSNR, RMSE, and SSIM over the entire 180-image synthetic test set.  For the two model-based methods, we tuned parameters to maximize average PSNR on this dataset: TV-Gnorm uses $(\alpha, \lambda) = (8, 10^{-3})$ following the notation of~\cite{wen2019primal}, and LR-WLS uses $(\tau, \beta) = (1.05,\, 3{\times}10^{-2})$ following~\cite{li2025cartoon}. For Joint-PnP~\cite{doi:10.1137/24M1677770} and LPR-Net~\cite{10.1007/978-3-031-92366-1_30}, we use the official pretrained models released by the respective authors; all three learning-based methods (including ours) are trained on synthetic data generated by the same procedure~\cite{doi:10.1137/24M1677770}.

\begin{table}[tbp]
\centering
\small
\setlength{\tabcolsep}{3.5pt}
\renewcommand{\arraystretch}{1.05}
\begin{tabular}{lccc|ccc}
\hline
\multirow{2}{*}{Method} & \multicolumn{3}{c|}{Cartoon} & \multicolumn{3}{c}{Texture} \\
\cline{2-7}
& PSNR $\uparrow$ & RMSE $\downarrow$ & SSIM $\uparrow$
& PSNR $\uparrow$ & RMSE $\downarrow$ & SSIM $\uparrow$ \\
\hline
TV-Gnorm \cite{wen2019primal}
 & 34.969 & 0.022 & 0.949
 & 21.008 & 0.095 & 0.264 \\
LR-WLS \cite{li2025cartoon}
 & 34.287 & 0.023 & 0.941
 & 32.187 & 0.034 & 0.858 \\
Joint-PnP \cite{doi:10.1137/24M1677770}
 & 34.810 & 0.026 & 0.934
 & 33.163 & 0.035 & 0.914 \\
LPR-Net \cite{10.1007/978-3-031-92366-1_30}
 & 38.463 & 0.015 & 0.986
 & 35.727 & 0.028 & 0.950 \\
Ours (NGVD)
 & \textbf{41.969} & \textbf{0.010} & \textbf{0.993}
 & \textbf{41.967} & \textbf{0.010} & \textbf{0.952} \\
\hline
\end{tabular}
\caption{Average PSNR~(dB), RMSE, and SSIM on the 180-image synthetic test dataset. Bold values denote the best results.}
\label{tab:full_comparison}
\end{table}

NGVD outperforms all baselines by a substantial margin.  Compared with the second-best method (LPR-Net), it improves average PSNR by 3.5~dB for cartoon and 6.2~dB for texture. TV-Gnorm achieves reasonable cartoon quality but severely under-separates texture (PSNR~21.0~dB), while LR-WLS provides a more balanced but still inferior decomposition.

For real-world natural images, we assess the proposed method on four challenging examples: Barbara,  Barcelona, Bee, and Bordeaux. Without ground-truth, evaluations are based on qualitative visual comparisons.

Figures~\ref{fig:natural_results} and~\ref{fig:natural_results2}
illustrate the results. Each subfigure includes the decomposed component
(cartoon or texture) with enlarged views of selected regions appended
below, enabling detailed comparisons of edge preservation, texture
isolation, and artifact reduction.
Our approach preserves sharp contours and smooth lines in cartoons---such
as clear facial outlines, building edges, and object shapes---while
accurately isolating fine-scale patterns like fabric weaves, soft bokeh, and biological fibers in textures. TV-Gnorm consistently
over-smooths, leading to blurred boundaries and diluted textures.
Joint-PnP misallocates textures to cartoons and causes edge artifacts.
LR-WLS retains residual textures at boundaries. LPR-Net achieves basic
separation but exhibits fuzzy edges and incomplete texture capture for
complex patterns. These observations are consistent across all four
images, confirming the robustness of the adaptive weighting mechanism to
diverse real-world content.

\begin{figure}[tbp]
\centering
\begingroup
  \presetSix
  \begin{tabular}{@{} p{\colw} p{\colw} p{\colw} p{\colw} p{\colw} p{\colw} @{}}
  \multirow{2}{*}{\raisebox{-0.5\height}{\cellimg{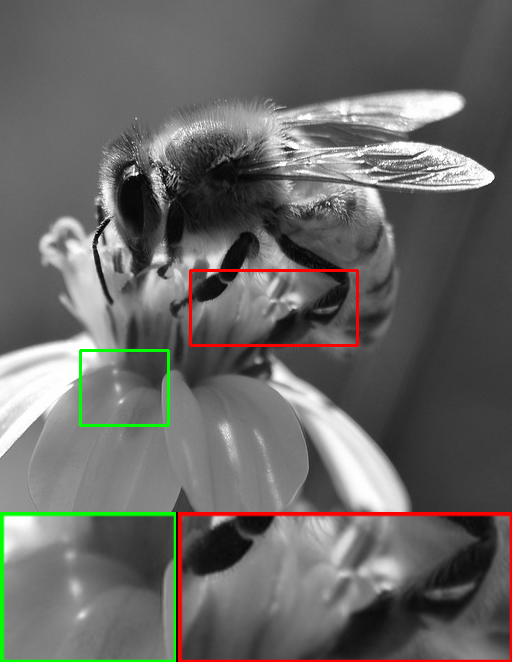}{}}} &
  \cellimg{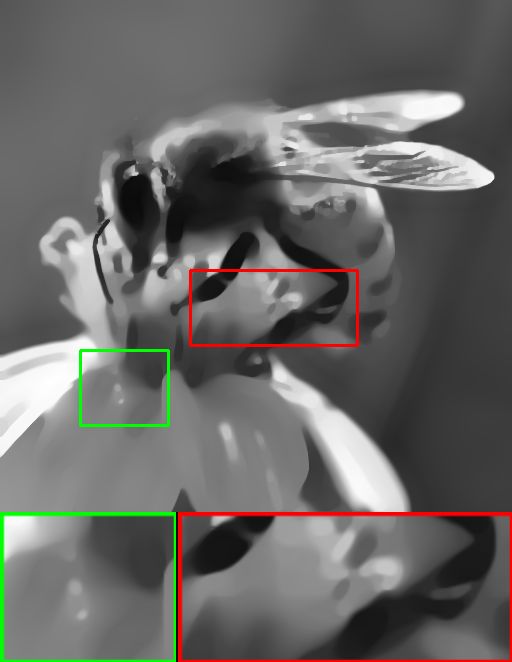}{} &
  \cellimg{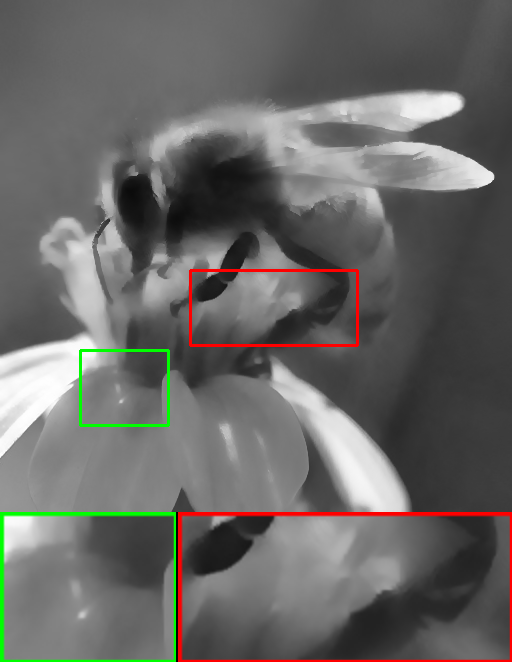}{} &
  \cellimg{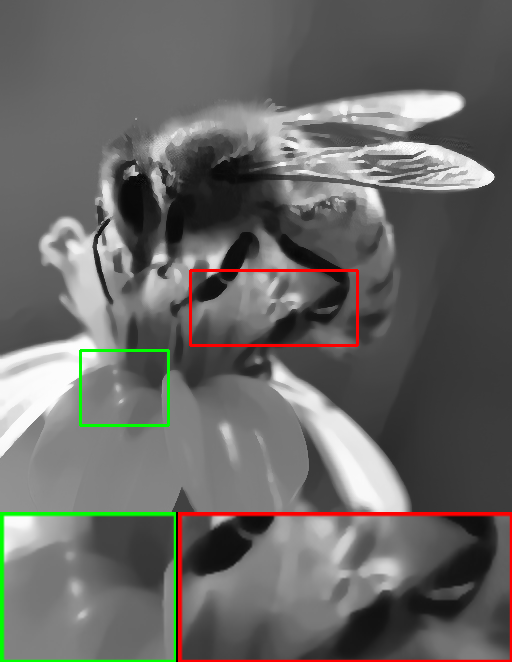}{} &
  \cellimg{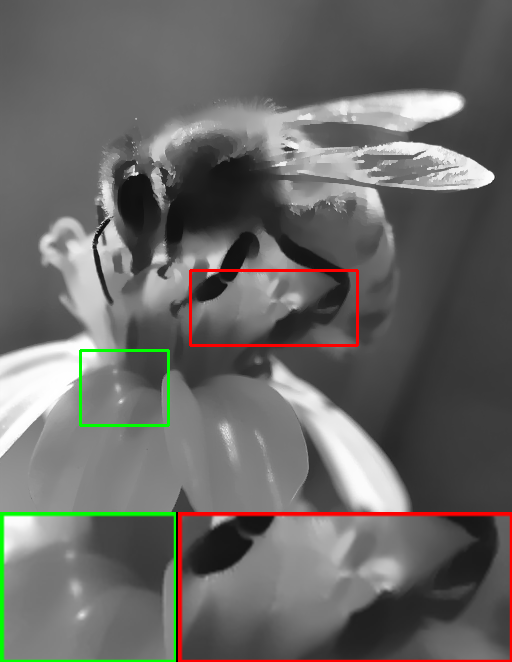}{} &
  \cellimg{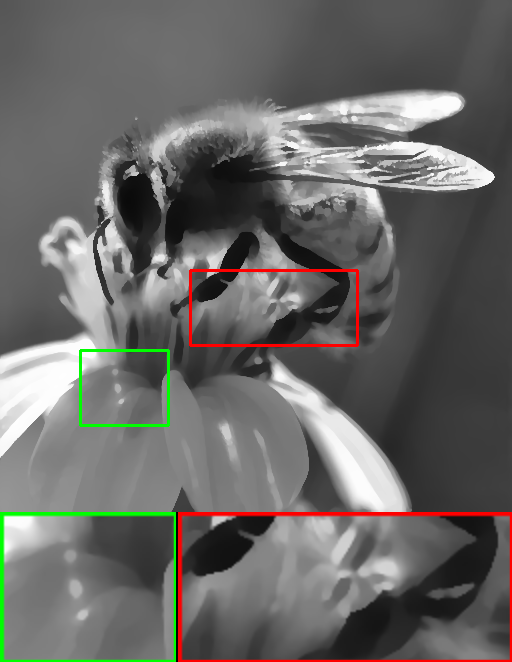}{} \\
  &
  \cellimg{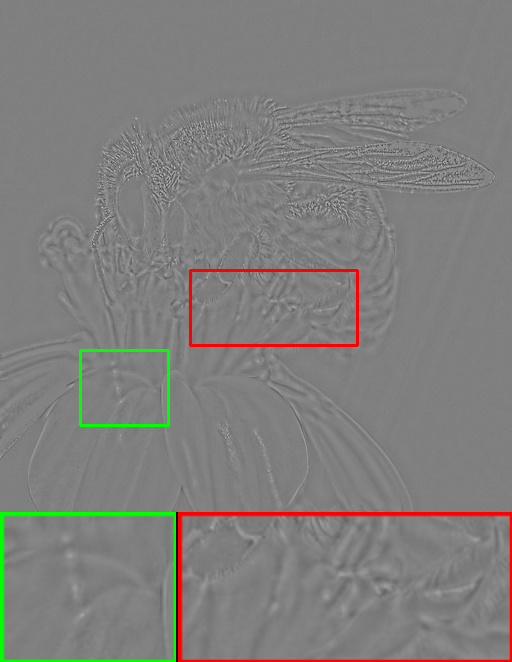}{} &
  \cellimg{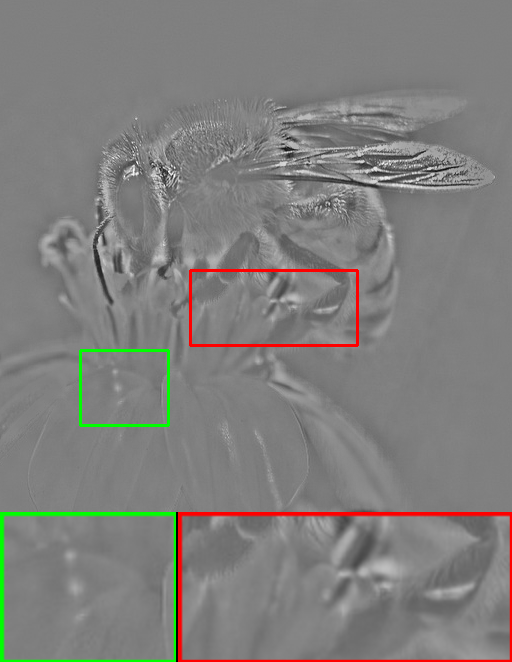}{} &
  \cellimg{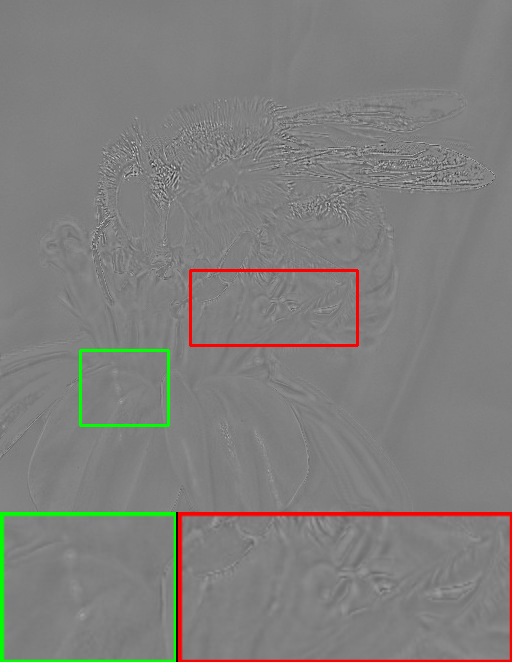}{} &
  \cellimg{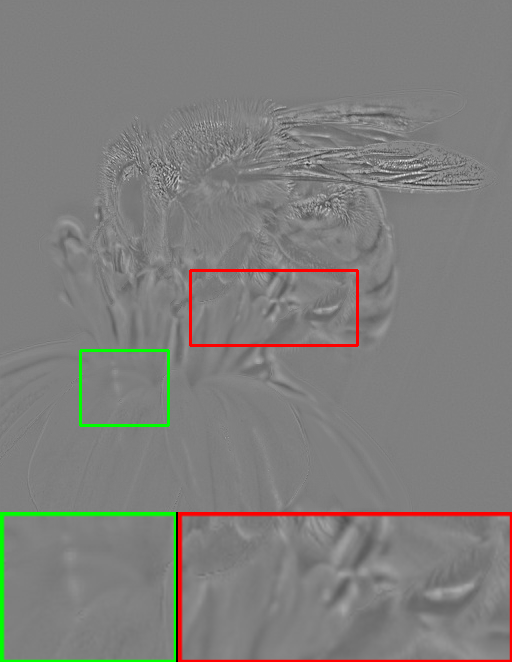}{} &
  \cellimg{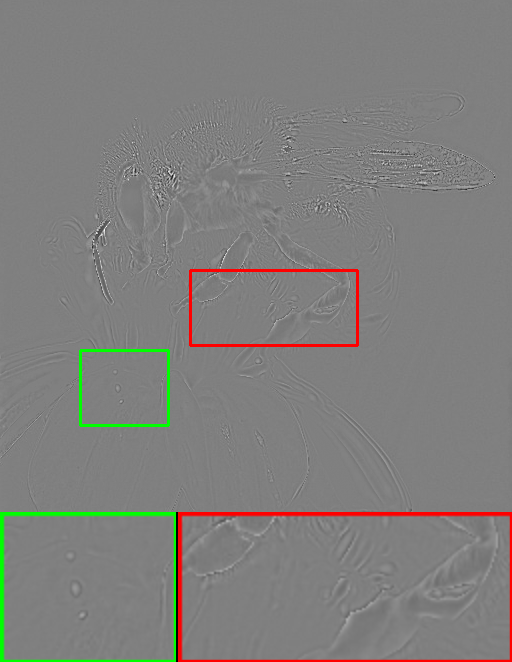}{} \\
  \end{tabular}
  \vspace{4pt}
  \begin{tabular}{@{} p{\colw} p{\colw} p{\colw} p{\colw} p{\colw} p{\colw} @{}}
  \multirow{2}{*}{\raisebox{-0.5\height}{\cellimg{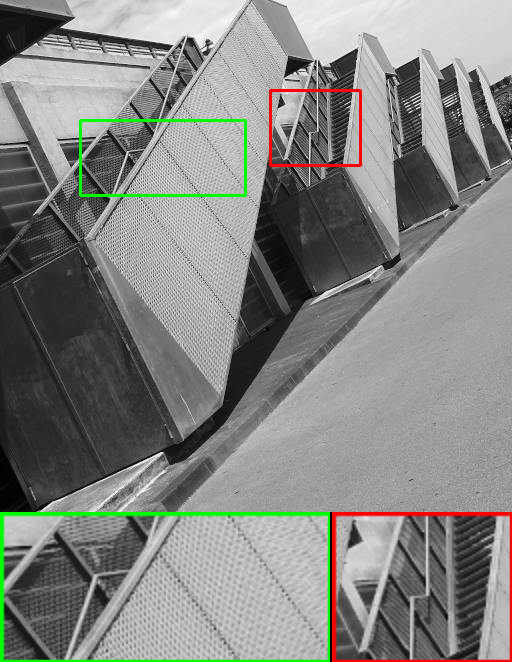}{}}} &
  \cellimg{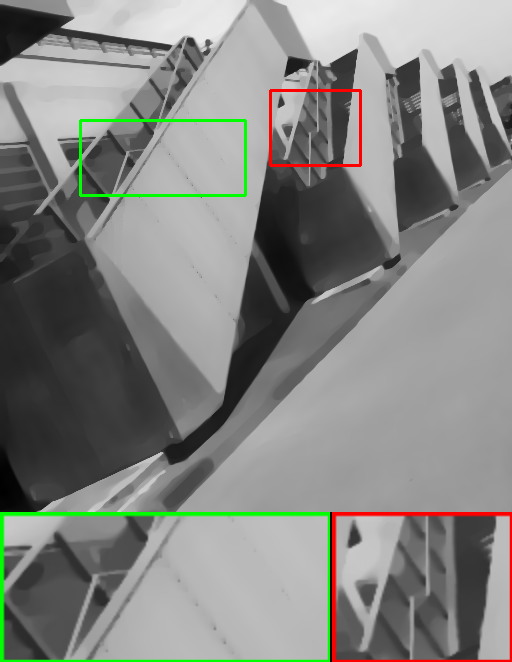}{} &
  \cellimg{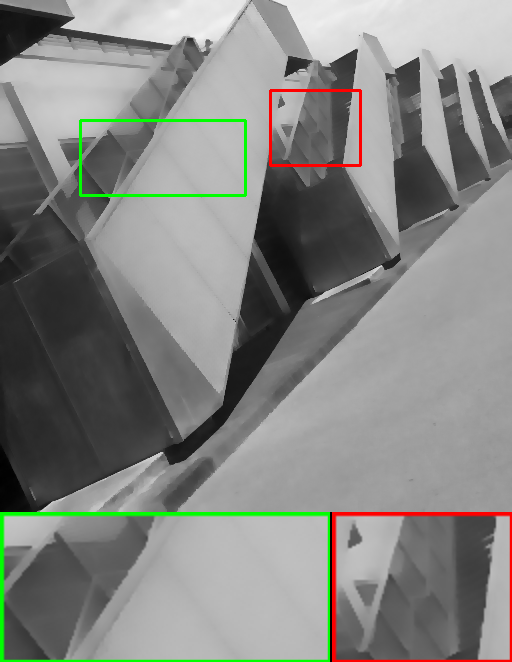}{} &
  \cellimg{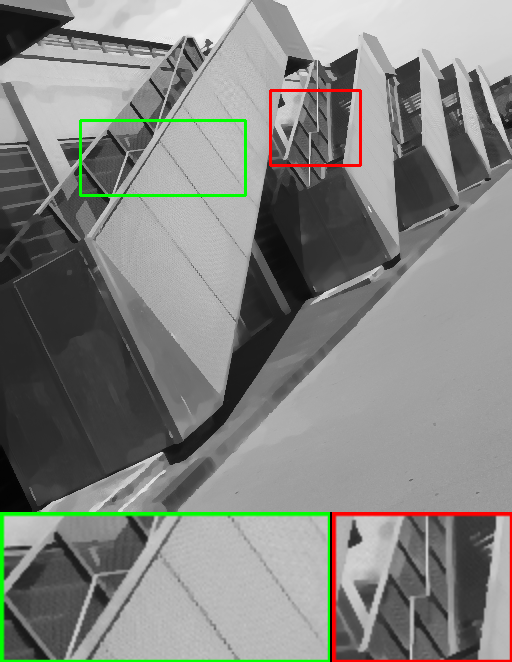}{} &
  \cellimg{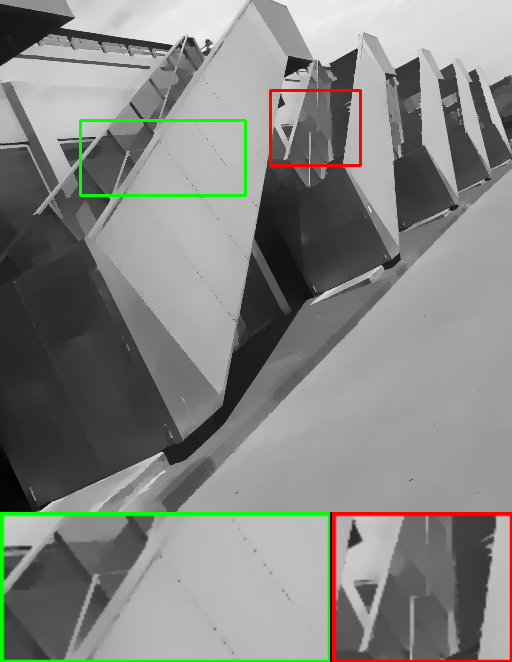}{} &
  \cellimg{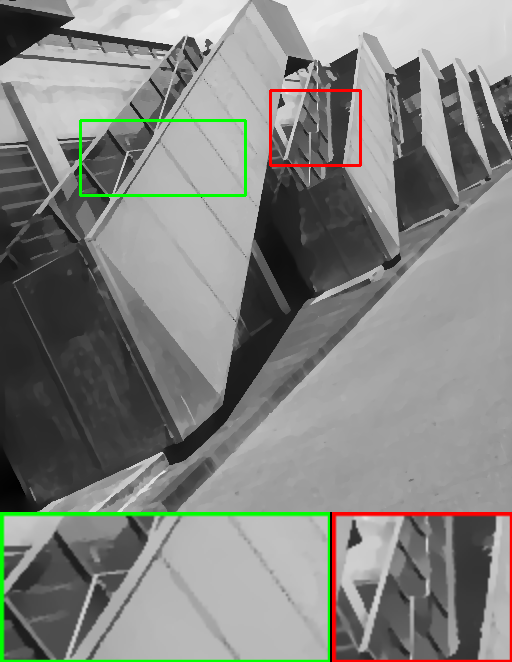}{} \\
  &
  \cellimg{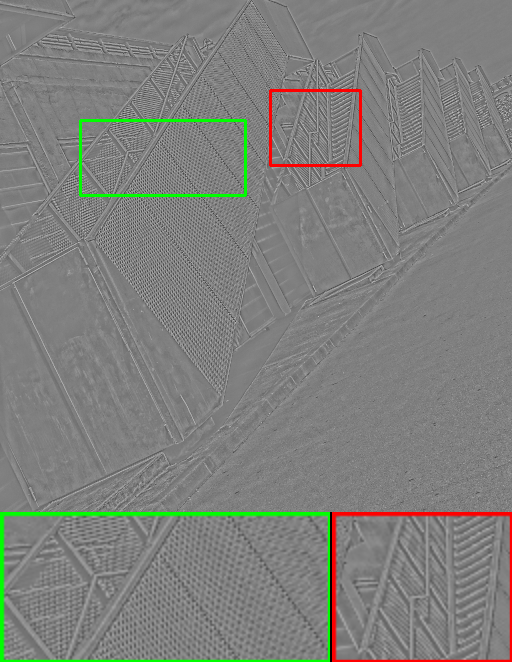}{} &
  \cellimg{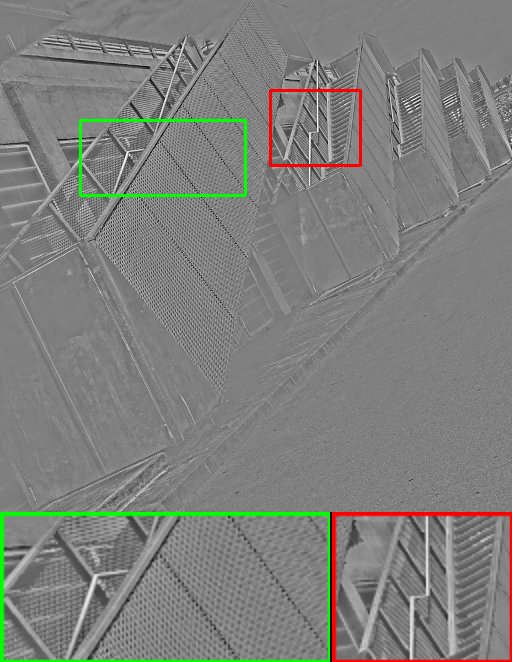}{} &
  \cellimg{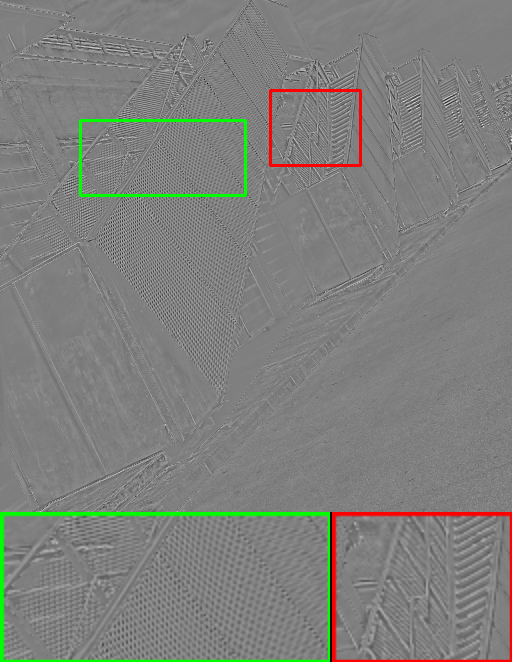}{} &
  \cellimg{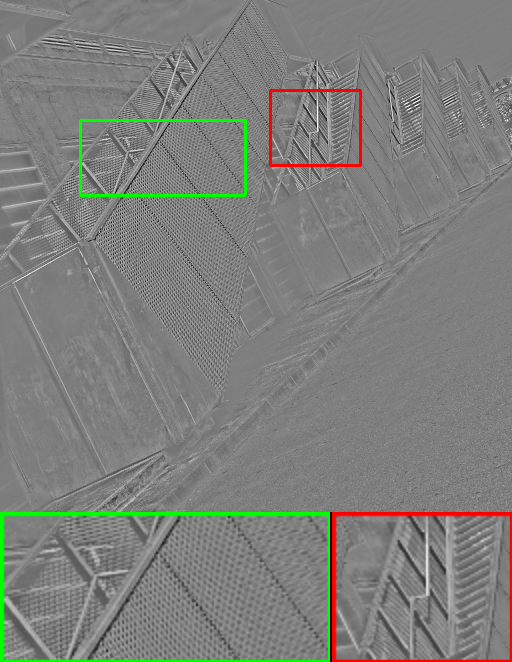}{} &
  \cellimg{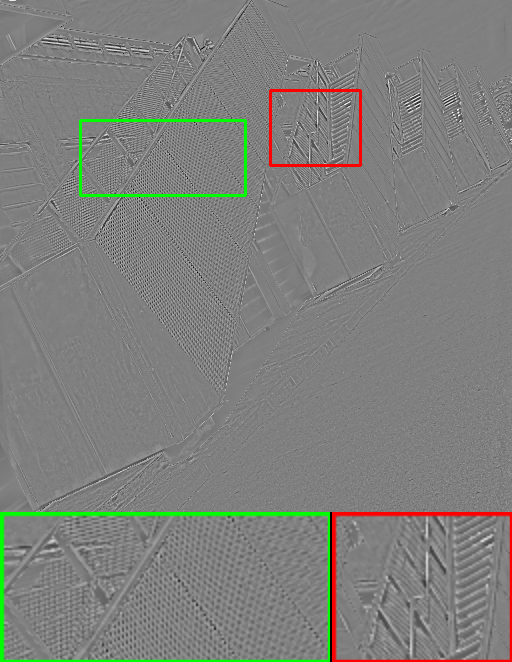}{} \\
  \end{tabular}
  \vspace{6pt}
  \begin{tabular}{@{} p{\colw} p{\colw} p{\colw} p{\colw} p{\colw} p{\colw} @{}}
  & \centering {\fontsize{7.0}{8.2}\selectfont\textsf{\mbox{TV-Gnorm \cite{wen2019primal}}}} &
  \centering {\fontsize{7.0}{8.2}\selectfont\textsf{\mbox{LR-WLS \cite{li2025cartoon}}}} &
  \centering {\fontsize{7.0}{8.2}\selectfont\textsf{\mbox{Joint-PnP \cite{doi:10.1137/24M1677770}}}} &
  \centering {\fontsize{7.0}{8.2}\selectfont\textsf{\mbox{LPR-Net \cite{10.1007/978-3-031-92366-1_30}}}} &
  \centering {\fontsize{7.0}{8.2}\selectfont\textsf{\mbox{Ours(NGVD)}}} \\
  \end{tabular}
  \caption{Real-world decomposition results on Bee and Bordeaux.
Left column: observed input (spanning two rows, with zoomed regions
integrated below). Top row per sample: cartoon component; bottom row:
texture component.}
  \label{fig:natural_results2}
\endgroup
\end{figure}

\subsection{Ablation Study on Iteration Counts and Subnetwork Learning Strategies}
\label{sec:ablation}

To further assess the impact of iterative weight updates and subnetwork learning strategies in our method, we conducted an ablation study on a simulated test dataset comprising 180 images, each of dimensions $128 \times 128$ pixels, generated following the protocol in Section~\ref{sec:exp_setup}. Evaluations focused on average PSNR, RMSE, and SSIM for the reconstructed cartoon and texture components.

\begin{table}
\centering
\small 
\setlength{\tabcolsep}{3.5pt} 
\renewcommand{\arraystretch}{1.05} 

\begin{tabular}{lccc|ccc}
\hline
\multirow{2}{*}{Configuration} & \multicolumn{3}{c|}{Cartoon} & \multicolumn{3}{c}{Texture} \\
\cline{2-7}
& PSNR $\uparrow$ & RMSE $\downarrow$ & SSIM $\uparrow$ & PSNR $\uparrow$ & RMSE $\downarrow$ & SSIM $\uparrow$ \\
\hline
\multicolumn{7}{c}{\textit{Iteration Count Ablation (Both Learned)}} \\
\hline
$k=1$  & 41.237 & 0.011 & 0.992 & 41.220 & 0.011 & \textbf{0.957} \\
$k=4$  & 41.344 & 0.011 & 0.993 & 41.342 & 0.011 & 0.951 \\
$k=8$  & \textbf{41.969} & \textbf{0.010} & \textbf{0.993} & \textbf{41.967} & \textbf{0.010} & 0.952 \\
$k=12$ & 41.581 & 0.010 & 0.993 & 41.581 & 0.010 & 0.949 \\
$k=16$ & 39.125 & 0.014 & 0.986 & 39.126 & 0.014 & 0.944 \\
\hline
\multicolumn{7}{c}{\textit{Subnetwork Learning Ablation ($k=8$)}} \\
\hline
Without $\Lambda_{\Theta_1}$ & 41.543 & 0.011 & 0.992 & 41.542 & 0.011 & \textbf{0.956} \\
Without $\mathcal{W}_{\Theta_2}$ & 33.170 & 0.025 & 0.931 & 33.152 & 0.025 & 0.888 \\
Absorbed $\lambda_1,\lambda_2$ into $\mathcal{W}_{\Theta_2}$ & 
35.057 & 0.021 & 0.954 & 34.952 & 0.021 & 0.909 \\
With $\Lambda_{\Theta_1}$ and $\mathcal{W}_{\Theta_2}$ & \textbf{41.969} & \textbf{0.010} & \textbf{0.993} & \textbf{41.967} & \textbf{0.010} & 0.952 \\
\hline
\end{tabular}

\caption{Average PSNR (dB), RMSE, and SSIM on the 180-image test dataset for varying iteration counts (with both subnetworks learned) and subnetwork learning strategies (at 8 iterations). Bold values denote the best results.}
\label{tab:ablation_results}
\end{table}

We first examined varying outer iteration counts ($k = 1, 4, 8, 12, 16$), with both subnetwo-rks---$\Lambda_{\Theta_1}$ for predicting global regularization parameters $\lambda_1$ and $\lambda_2$, and $\mathcal{W}_{\Theta_2}$ for spatially adaptive weights---fully learned. This identifies the iteration count optimizing the trade-off between computational efficiency and decomposition quality.

Based on results in Table~\ref{tab:ablation_results}, we selected $k=8$ as optimal and performed additional ablations by disabling learning in one subnetwork while maintaining it in the other. Configurations included:
\begin{enumerate}
    \item Without learning $\Lambda_{\Theta_1}$ (fixed $\lambda_1 = 1$, $\lambda_2 = 0.2$) and with learning $\mathcal{W}_{\Theta_2}$.
    \item With learning $\Lambda_{\Theta_1}$ and without learning $\mathcal{W}_{\Theta_2}$ (using identity matrices for $W_1$ and $W_2$).
    \item Absorbing $\lambda_1, \lambda_2$ into $\mathcal{W}_{\Theta_2}$
    : $\mathcal{W}_{\Theta_2}$ is retrained from
    scratch with no separate global scalar, forcing it to simultaneously encode
    global fidelity/regularization trade-offs and local spatial structure.
\end{enumerate}
These were compared against the full model (both subnetworks learned) to quantify individual contributions.

Table~\ref{tab:ablation_results} summarizes the results. Metrics improve progressively up to $k=8$,
after which performance declines at $k=12$ and drops more sharply at $k=16$.
This behavior is characteristic of truncated unrolling: training with more
unrolled steps increases the depth of the computational graph, making
optimization harder and more prone to overfitting the training trajectories.
We therefore adopt $k=8$ as the operating point that best balances
decomposition quality and training stability.

In the subnetwork ablations at $k=8$, disabling $\Lambda_{\Theta_1}$ results in slight degradation, with PSNR falling 0.426~dB (cartoon) and 0.425~dB (texture), underscoring the role of adaptive global regularization in maintaining equilibrium between data fidelity and smoothness. Conversely, omitting $\mathcal{W}_{\Theta_2}$ causes marked deterioration, with PSNR dropping 8.799~dB (cartoon) and 8.815~dB (texture), revealing that data-driven spatial weights are essential for discerning heterogeneous local structures, preventing texture bleed and edge artifacts. These observations affirm the synergistic interplay of both subnetworks in the iterative scheme, fostering enhanced adaptability and precision in decomposition tasks.

Although absorbing $\lambda_1, \lambda_2$ into $W_1, W_2$ is algebraically
possible, the separation is both theoretically necessary and empirically
critical. The convergence analysis in Section~\ref{sec:theory} requires
$\omega_{\min}I \preceq W_i \preceq \omega_{\max}I$ with $\omega_{\max} < 1$,
which merging $\lambda_i$ into $W_i$ would violate. The scalars $\lambda_1,
\lambda_2$ govern the global fidelity/regularization balance while $W_1, W_2$
encode local pixel-wise structure---conflating the two roles forces
$\mathcal{W}_{\Theta_2}$ to span vastly different magnitudes within a bounded
output range, causing gradient imbalance and training instability.
Empirically, the ``Absorbed $\lambda_1,\lambda_2$ into $\mathcal{W}_{\Theta_2}$''
row in Table~\ref{tab:ablation_results} directly tests this scenario:
with $\Lambda_{\Theta_1}$ removed entirely and $\mathcal{W}_{\Theta_2}$
retrained from scratch, PSNR drops ${\approx}6.9$\,dB (cartoon) and ${\approx}7.0$\,dB (texture)---far exceeding the mild 0.43\,dB drop when
$\lambda_1, \lambda_2$ are merely fixed as separate scalars. This confirms
that the separation of $\Lambda_{\Theta_1}$ and $\mathcal{W}_{\Theta_2}$ is a
structural necessity.

\subsection{Empirical validation of the theoretical results in Sec. \ref{sec:tr}}
\label{sec:conv_verify}

We first empirically validate the linear convergence established in 
Theorem~\ref{thm:exist_converge_main}. Algorithm~\ref{alg:test} is run for 
$K_{\max}=50$ outer iterations on the three images 
from Figure~\ref{fig:decomp-grid}, with the final iterate 
serving as a surrogate fixed point $x^\star\approx \widehat{x}_{K_{\max}}$. 
Figure~\ref{fig:convergence}(a) shows the empirical contraction 
factor
\begin{equation}\label{eq:Qk}
  Q_k \;:=\; 
  \frac{\|\widehat{x}_{k+1}-x^\star\|_2}{\|\widehat{x}_k-x^\star\|_2}\,,
\end{equation}
which stabilises in the range $[0.85,\,0.95]$, confirming that the 
iteration is contractive with a rate consistent with the bound in 
Theorem~\ref{thm:exist_converge_main}. The corresponding logarithmic fixed-point error 
$\log_{10}\|\widehat{x}_k-x^\star\|_2$ is plotted in 
Figure~\ref{fig:convergence}(b); its monotone decay at a rate 
governed by~$Q_k$ corroborates Q-linear convergence.

Figure~7(c) complements the fixed-point graphs by monitoring the PSNR values of the decompositions obtained along the 50 iterations. In particular, we plot, as a function of $k$, the PSNR values of the joint decomposition $\widehat{g}_k:=\mathcal{D}\widehat{x}_k = (\widehat{c}_k,\widehat{t}_k)$. 
Note that graphs in Figure \ref{fig:convergence}(a)-(b) 
describe convergence of $\widehat{x}_k$ to $x^\star$, whereas those in Figure \ref{fig:convergence}(c) measure $\widehat{g}_k$ to the target 
ground truth decomposition $g=(c,t)$. It clearly emerges from Figure \ref{fig:convergence}(c) that the PSNR values stabilise rapidly after $k=8$ iterations,
with a maximum drop of $0.56$\,dB at $k=50$, confirming that $x^\star$ is very close to the ground truth in practice and that $k=8$ is a safe and principled operating point even when the iteration is continued beyond the training horizon.

\begin{figure}[tbp]
  \centering
  \includegraphics[width=\textwidth]{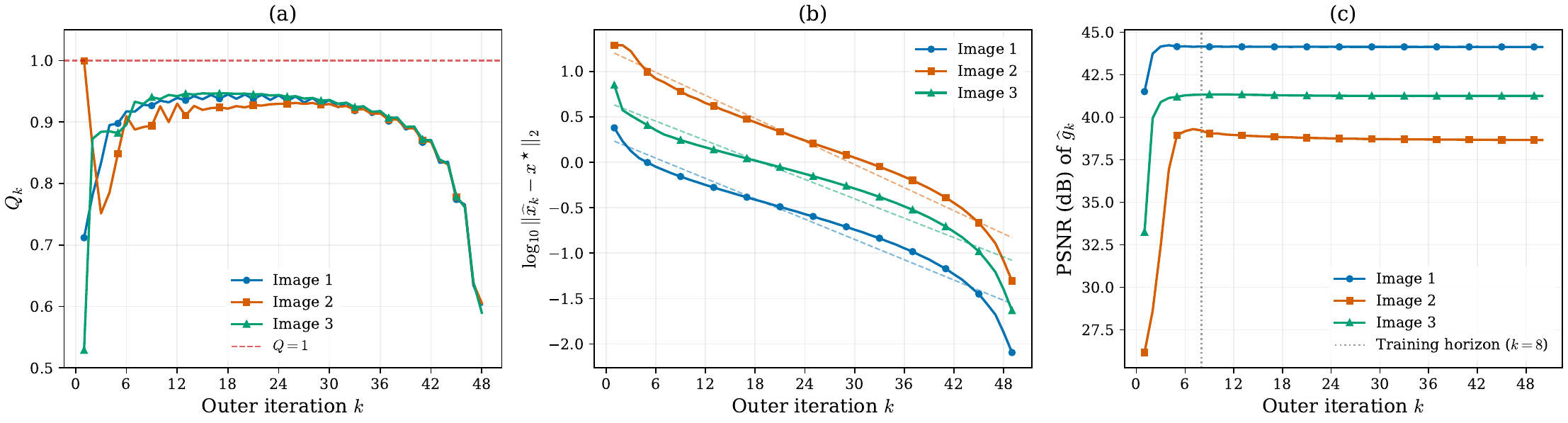}
  \caption{Convergence behavior of Algorithm~\ref{alg:test} applied to images from Figure~\ref{fig:decomp-grid}. \textbf{(a)}~Contraction factor $Q_k$ vs. $k$: values stabilise below $1$. \textbf{(b)}~Fixed-point logarithmic error 
  $\log_{10}\|\widehat{x}_k-x^\star\|_2$ vs.\ $k$ (dashed: linear fits).
  (c)~PSNR of $\widehat{g}_k = (\widehat{c}_k,\widehat{t}_k)$ vs.\ $k$: training horizon $k=8$
(dotted vertical line).
  }
  \label{fig:convergence}
\end{figure}

To empirically validate Proposition~\ref{prop:stab}, we use again the three synthetic sample test images from Figure~\ref{fig:decomp-grid}, all of size $128 \times 128$. 
The weight matrices $(W_1,W_2)$ are fixed from the clean-image
decomposition with $\omega_{\min}=0.01$, and the linear system~ \eqref{eq:normalA_main} is solved for noisy inputs
$\hat{f} = f + \epsilon$ with noise $\epsilon \sim \mathcal{N}(0,\sigma^2 I)$ of standard deviation in the range $\sigma \in [0.002,\,0.02]$.
Numerical computation leads to $\|S\|\approx 2.99$ and
$\sigma_{\min}(M)\approx 0.045$, hence $\mathrm{cond}(M)\approx 67$
and the Lipschitz constant $\|S\|/\alpha\approx 1497$.
Figure~\ref{fig:stability}(a) confirms that the empirical ratio
$\|\Delta x\|_2/\|\Delta f\|_2$ remains strictly below the theoretical
bound for all sample images and noise levels, thus validating Proposition~\ref{prop:stab}.
The empirical ratio stabilises near one for images~1 and~3, and increases toward $3$ for image~2, indicating that the theoretical upper bound is valid but
conservative: $\alpha = \sigma_{\min}^2(M)$ is governed by the
minimum admissible weight $\omega_{\min}$, while in practice the
learned weights concentrate well above this value.
Figure~\ref{fig:stability}(b) shows that the cartoon PSNR drops by at most $0.3$\,dB
at $\sigma=0.02$, while the texture PSNR decreases more substantially,
consistent with the oscillatory and high-frequency nature of the texture
component.

\begin{figure}[tbp]
  \centering
  \includegraphics[width=\textwidth]{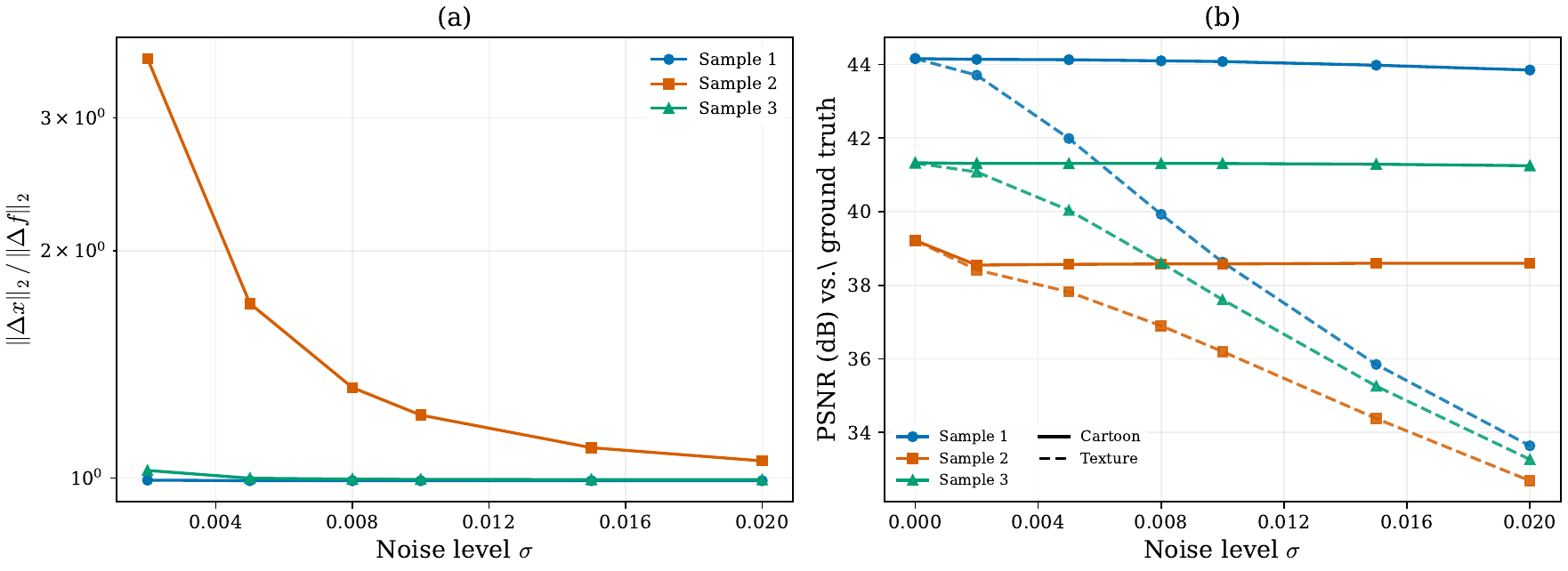}
  \caption{Numerical validation of Proposition~\ref{prop:stab}.
(a)~Empirical ratio $\|\Delta x\|_2/\|\Delta f\|_2$ (log scale) vs.\
noise level $\sigma$; note that the theoretical upper bound  is $\|S\|/\alpha\approx 1497$.
(b)~Decomposition PSNR vs.\ $\sigma$: cartoon (solid) and texture
(dashed).}
  \label{fig:stability}
\end{figure}

\section{Conclusion}
\label{sec:con}
In this paper, we introduced the Neural Guided Variational Decomposition (NGVD) framework, a novel approach to cartoon–texture separation that bridges the gap between classical variational models and deep learning. By employing spatially adaptive, pixel-wise weights within a quadratic formulation, we demonstrated that it is possible to maintain the computational efficiency of linear systems while capturing the complex structural heterogeneity of natural images. Our work provided two distinct pathways for weight estimation: a supervised, data-driven variant utilizing an MLP and a lightweight U-Net, and a robust model-based probabilistic estimator for training-free applications.

Theoretically, we established the mathematical rigor of the NGVD approach by framing the iterative refinement scheme as a fixed-point map. We provided formal proofs for the uniqueness and conditioning of the inner solves, the existence of outer fixed points, and, crucially, a verifiable contractivity condition that ensures convergence. 
Furthermore, our Lipschitz stability analysis confirms the framework's practical resilience against measurement perturbations and noise. Extensive numerical experiments validate these theoretical findings, showing that NGVD consistently outperforms classical and recent state-of-the-art methods in terms of decomposition quality and edge preservation.

This framework opens several promising avenues for future research. While we focused on the cartoon-texture problem, future work could explore  the extension of the neural-guided weights to handle multi-component decomposition of images and signals into three or more constituents.

\appendix
\section{Proofs and technical estimates}
\label{app:proofs}

This appendix contains proofs of Proposition \ref{prop:MAP}, Lemma \ref{Alowerbound} and Theorems~\ref{thm:lipschitz_main}--\ref{thm:exist_converge_main}, together with an auxiliary Lemma \ref{lem:perturbapx}.

\subsection{Proof of Proposition \ref{prop:MAP}}
\begin{proof}
We employ a MAP estimation approach to estimate the components $c$ and $\xi$:
\begin{align}
\{\widehat{c},\widehat{\xi}\} 
&= \arg \max_{c,\xi} p\left(c,\xi\mid f\right)
= \argmin_{c,\xi} -\ln \left( p\left(c,\xi\mid f\right) \right) 
= \argmin_{c,\xi} -\ln 
\frac{p\left(f \mid c,\xi\right) \, p\left(c,\xi\right)}{p\left(f\right)}
\nonumber\\
&= \argmin_{c,\xi} \left[ 
-\ln p\left(f \mid c,\xi\right) 
-\ln p\left(c,\xi\right) 
+\cancel{\ln p\left(f\right)} 
\right]
\label{eq:MAPPP}
\end{align}
where we used Bayes' theorem and omitted the log-evidence term $\ln p(f)$ as it is constant with respect to the optimization variables $c$ and $\xi$. 

The explicit expression for the negative log-likelihood $-\ln p(f \mid c,\xi)$ is derived by considering the generative model $f = c + t + r = c + \mathrm{div}(\xi) + r$. Noting that $c + \mathrm{div}(\xi)$ is deterministic when conditioning on $c$ and $\xi$, the likelihood satisfies:
\begin{equation}
p(f \mid c,\xi) = p(c + \mathrm{div}(\xi) + r \mid c,\xi) = p_r(f - c - \mathrm{div}(\xi) \mid c,\xi) = p_r(f - c - \mathrm{div}(\xi)) ,
\label{eq:negLL}
\end{equation}
where the second equality reflects the translation property of the probability density, and the final equality follows from the mutual independence of the components $r$ and $(c,\xi)$ as stated in assumption \ref{eq:H1}.

After noting that the negative log-distribution of a $n$-variate Gaussian random vector $z$ with zero-mean and covariance matrix $\Sigma_z$ reads
\begin{equation}
-\ln p(z) 
\,\;{=}\:
-\ln \left[ \frac{1}{\sqrt{(2 \pi)^n | \Sigma_z |}} \, \mathrm{exp} \left(
-\frac{1}{2} \, z^{\top} \Sigma_z^{-1} z \right) \right]
\,\;{=}\;\:
\frac{1}{2} \left\|  z \right\|_{\Sigma_z^{-1}}^2 + \, \text{const} 
\, ,
\label{eq:nllg}
\end{equation}
with const not depending on $z$, and recalling assumption \ref{eq:H2} on the distribution of $r$, the negative log-likelihood in \eqref{eq:negLL} takes the form
\begin{equation}
- \ln p\left(f \mid c,\xi \right) = 
\frac{1}{2 \sigma_r^2} \left\| f - c - \mathrm{div}(\xi) \right\|_2^2 
\, + \; \text{const} \, .
\label{eq:LIK}
\end{equation}

Then, based again on assumption \ref{eq:H1} that the two sought components $c$ and $\xi$ are mutually independent, and also recalling the last two assumptions \ref{eq:H3}, \ref{eq:H4_1} on the distributions of $c$ and $\xi$, we find that the negative log-prior in \eqref{eq:MAPPP} reads
\begin{equation}
-\ln p(c,\xi) 
\,\;{=}\;
-\ln p(c) - \ln p(\xi)
\,\;{=}\;\, 
\frac{1}{2} \left\|  \nabla c \right\|_{\Sigma_{ c}^{-1}}^2
+
\frac{1}{2} \left\|  \xi \right\|_{\Sigma_{\xi}^{-1}}^2
\, + \; \text{const} \, .
\label{eq:PRIOR}
\end{equation}
Plugging \eqref{eq:LIK}--\eqref{eq:PRIOR} into the MAP formula \eqref{eq:MAPPP} and neglecting the constants, we get
\begin{equation}
\{\widehat{c},\widehat{\xi}\} 
= \argmin_{c,\xi} \left\{
\frac{1}{2 \sigma_r^2} \left\| f - c - \mathrm{div}(\xi) \right\|_2^2
+\frac{1}{2} \left\|  \nabla c \right\|_{\Sigma_{ c}^{-1}}^2
+\frac{1}{2} \left\|  \xi \right\|_{\Sigma_{\xi}^{-1}}^2
\right\}
\label{eq:MAP2}
\end{equation}
Noting that the two minimum variances $\underline{\sigma}_{\, c}^2$, $\underline{\sigma}_{\, \xi}^2$ defined in \eqref{eq:MAPpars} 
are positive by assumption, and that the two ``normalized'' covariance 
matrices $\underline{\Sigma}_{\, c}$, $\underline{\Sigma}_{\,\xi}$ defined in \eqref{eq:MAPsigm}
\begin{equation}
\underline{\Sigma}_{\, c} := \frac{1}{\underline{\sigma}_{\, c}^2} \, \Sigma_{ c} 
\,\, , 
\quad\;\;\; 
\underline{\Sigma}_{\,\xi} = \frac{1}{\underline{\sigma}_{\,\xi}^2} \Sigma_{\xi} 
\,\, ,
\end{equation}
whose diagonal elements are clearly all greater than or equal to 1, 
\eqref{eq:MAP2} can be equivalently written as
\begin{equation}
\{\widehat{c},\widehat{\xi}\} 
= \argmin_{c,\xi} \left\{
\frac{1}{2 \sigma_r^2} \left\| f - c - \mathrm{div}(\xi) \right\|_2^2
+\frac{1}{2 \underline{\sigma}_{\, c}^2} 
  \left\|  \nabla c \right\|_{\underline{\Sigma}_{\, c}^{-1}}^2
+\frac{1}{2 \underline{\sigma}_{\,\xi}^2} 
  \left\|  \xi \right\|_{\underline{\Sigma}_{\,\xi}^{-1}}^2
\right\} \, .
\label{eq:MAP3}
\end{equation}
Finally, multiplying the cost function in \eqref{eq:MAP3} by the positive scalar $\sigma_r^2$ and introducing the variables $\lambda_1$, $\lambda_2$ defined in \eqref{eq:MAPpars}, we obtain the variational model in \eqref{eq:our}.
\end{proof}

\subsection{Proof of Lemma \ref{Alowerbound}}

\begin{proof}
Using $W_i\succeq\omega_{\min}I$, we have
\begin{align*}
x^\top A(W_1,W_2) x
&= \|Sx\|_2^2 + \lambda_1 (Gx)^\top W_1 (Gx) + \lambda_2 (Rx)^\top W_2 (Rx)\\
&\ge \|Sx\|_2^2 + \lambda_1\omega_{\min}\|Gx\|_2^2 + \lambda_2\omega_{\min}\|Rx\|_2^2\\
&= \|\mathcal M x\|_2^2 \\
&\ge \sigma_{\min}^2(\mathcal M)\|x\|_2^2.
\end{align*}
From Prop.\ref{prop1} $A(W_1,W_2)$ is symmetric, positive definite, thus invertible. Then, since  the minimum eigenvalue satisfies \(\lambda_{\min}(A(W_1,W_2)) \ge \sigma_{\min}^2(\mathcal M)=\alpha>0\),
then the inverse-norm bound follows 
\[
\|A(W_1,W_2)^{-1}\| = \frac{1}{\lambda_{\min}(A(W_1,W_2))} \le \frac{1}{\alpha}.
\]
\end{proof}

\subsection{Proof of Theorem~\ref{thm:lipschitz_main}}
Before giving the proof of the theorem, we first propose the following auxiliary lemma. We quantify how $\mathcal{T}_{\theta}(x)$ depends on changes in $\mathcal{W}_{\Theta}(x)$.
\begin{lemma}\label{lem:perturbapx}
Let $\mathcal{W}_{\Theta}(x_1)=(W_1,W_2)$ and ${\mathcal{W}}_{\Theta}(x_2)=(\widetilde W_1,\widetilde W_2)$ be two admissible weight pairs. Then
\begin{equation}\label{eq:perturb_bound}
\|\mathcal{T}_{\theta}(x_1)-\mathcal{T}_{\theta}(x_2)\|_2 \le \frac{\lambda_1\|G\|^2\|\widetilde W_1-W_1\| + \lambda_2\|R\|^2\|\widetilde W_2-W_2\|}{\alpha}\,\|\mathcal{T}_{\theta}(x_2)\|_2.
\end{equation}
\end{lemma}

\begin{proof}
From normal equations, we have
\[
A(\mathcal{W}_{\Theta}(x_1))\mathcal{T}_{\theta}(x_1)=b=A(\mathcal{W}_{\Theta}(x_2))\mathcal{T}_{\theta}(x_2).
\]
Subtract to obtain
\[
A(\mathcal{W}_{\Theta}(x_1))\big(\mathcal{T}_{\theta}(x_1)-\mathcal{T}_{\theta}(x_2)\big) = \big(A(\mathcal{W}_{\Theta}(x_2))-A(\mathcal{W}_{\Theta}(x_1))\big) \mathcal{T}_{\theta}(x_2).
\]
Hence
\[
\mathcal{T}_{\theta}(x_1)-\mathcal{T}_{\theta}(x_2) = (A(\mathcal{W}_{\Theta}(x_1)))^{-1}\big(\lambda_1 G^\top(\widetilde W_1-W_1)G + \lambda_2 R^\top(\widetilde W_2-W_2)R\big)\mathcal{T}_{\theta}(x_2).
\]
Taking norms and using $\|A(\mathcal{W}_{\Theta}(x_1))^{-1}\|\le 1/\alpha$ and $\|G^\top(\widetilde W_1-W_1)G\|\le \|G\|^2\|\widetilde W_1-W_1\|$ yields \eqref{eq:perturb_bound}.
\end{proof}
Now, we give the proof of Theorem~\ref{thm:lipschitz_main}.
\begin{proof}
Let $x,y\in\mathcal B$ and denote $(W_1,W_2)=\mathcal W_{\Theta}(x)$, $(\widetilde W_1,\widetilde W_2)=\mathcal W_{\Theta}(y)$. For \(i=1,2,\) we have
\[
\|\widetilde W_i-W_i\| \le \|\mathcal W_{\Theta}(x)-\mathcal W_{\Theta}(y)\| \le L_{\mathcal W}\|x-y\|_2.
\]
Based on \eqref{eq:AinvL}, we obtain the bound $\|\mathcal{T}_{\theta}(y)\|_2 \le \|A(\mathcal{W}_{\Theta}(y))^{-1}\|\|b\|_2 \le \|b\|_2/\alpha$. Using Lemma \ref{lem:perturbapx} and $\|b\|_2=\|S^\top f\|_2\le\|S\|\,\|f\|_2$, we have
\[
\|\mathcal{T}_{\theta}(x)-\mathcal{T}_{\theta}(y)\|_2 \le \frac{(\lambda_1\|G\|^2+\lambda_2\|R\|^2)\,L_{\mathcal W}\,\|S\|\,\|f\|_2}{\alpha^2}\,\|x-y\|_2.
\]
This proves the Lipschitz bound.
\end{proof}

\subsection{Proof of Theorem~\ref{thm:exist_converge_main}}
\begin{proof}
\textbf{Invariant ball and existence.} For any admissible $\mathcal{W}_{\Theta}(x)$ we have
\[
\|\mathcal{T}_{\theta}(x)\|_2 = \|A(\mathcal{W}_{\Theta}(x))^{-1}b\| \le \|A(\mathcal{W}_{\Theta}(x))^{-1}\|\,\|b\| \le \frac{\|S\|\,\|f\|_2}{\alpha} =: r \, ,
\]
so $\mathcal T(\mathcal B)\subset\mathcal B$. Since $\mathcal T$ is continuous on $\mathcal B$ (Theorem~\ref{thm:lipschitz_main}), Brouwer's fixed-point theorem implies existence of at least one fixed point in $\mathcal B$.

\textbf{Contractivity and uniqueness.} Since the explicit upper bound $\mathcal{Q}$ in Theorem~\ref{thm:lipschitz_main} satisfies $\mathcal{Q}<1$ by choosing proper $c_i$, then $\mathcal T$ is a contraction on $\mathcal B$ and Banach's fixed-point theorem yields a unique fixed point $x_\star$ in $\mathcal B$ and linear convergence $\|x_k-x_\star\|\le \mathcal{Q}^k\|x_0-x_\star\|$.
This completes the proof.
\end{proof}

\bibliographystyle{siamplain}
\bibliography{references}
\end{document}